\documentclass{amsart}

\usepackage{amssymb}
\usepackage{amsmath}
\usepackage{verbatim}
\usepackage{url}
\usepackage{graphicx}
\usepackage{color}

\usepackage{amsaddr}

\usepackage{mathrsfs}

\newcommand{\eps}{\varepsilon}

\newcommand{\opnm}{\operatorname}

\newcommand{\Bff}{\mathbf}
\newcommand{\BigO}{\mathcal{O}}
\newcommand{\Spec}{\operatorname{Spec}}
\newcommand{\tr}{\operatorname{tr}}
\newcommand{\Op}{\operatorname{Op}}
\newcommand{\ii}{\textnormal{i}}
\newcommand{\dd}{\textnormal{d}}
\newcommand{\ee}{\textnormal{e}}

\newcommand\xqed[1]{%
  \leavevmode\unskip\penalty9999 \hbox{}\nobreak\hfill
  \quad\hbox{#1}}

\newtheorem{theorem}{Theorem}%[section]
\newtheorem{proposition}[theorem]{Proposition}

\newtheorem{lemma}[theorem]{Lemma}
\theoremstyle{definition}
\newtheorem{xdefinition}[theorem]{Definition}
\newenvironment{definition}{\begin{xdefinition}}{\xqed{$\triangle$}\end{xdefinition}}
\theoremstyle{remark}
\newtheorem{xremark}[theorem]{Remark}
\newenvironment{remark}{\begin{xremark}}{\xqed{$\triangle$}\end{xremark}}
\newtheorem{xexample}[theorem]{Example}
\newenvironment{example}{\begin{xexample}}{\xqed{$\triangle$}\end{xexample}}
\newtheorem*{acknowledgements}{Acknowledgements}

\numberwithin{equation}{section}
\numberwithin{theorem}{section}
\numberwithin{figure}{section}

\newcommand{\fnJ}[1]{}

\title{The elliptic evolution of non-self-adjoint degree-2 Hamiltonians}

\author{Joe Viola}

\email{Joseph.Viola@univ-nantes.fr}

\address{Laboratoire de Math\'{e}matiques Jean Leray \\ 2 rue de la Houssini\`{e}re \\ Universit\'{e} de Nantes \\ BP 92208 F-44322 Nantes Cedex 3}

\begin{document}

\begin{abstract}
We study the relationship between the classical Hamilton flow and the quantum Schr\"odinger evolution where the Hamiltonian is a degree-2 complex-valued polynomial. When the flow obeys a strict positivity condition equivalent to compactness of the evolution operator, we find geometric expressions for the $L^2$ operator norm and a singular-value decomposition of the Schr\"odinger evolution, using the Hamilton flow. The flow also gives a geometric composition law for these operators, which correspond to a large class of integral operators with nondegenerate Gaussian kernels.
\end{abstract}

\maketitle

\section{Introduction}

We study the Schr\"odinger evolution $\exp(-\ii P)$ where $P$ is the Weyl quantization (Definition \ref{def_Weyl}) of a certain type of degree-2 polynomial. The primary goal of this work is to identify the norm of $\exp(-\ii P)$ as an operator on $L^2(\Bbb{R}^n)$ using the Hamilton flow of its symbol (Theorems \ref{thm_norm_quadratic} and \ref{thm_norm_deg2}), though what we obtain is in fact a decomposition of singular-value type (Theorem \ref{thm_SVD}). We also show that the class of Schr\"odinger evolution operators considered here and in \cite{Aleman_Viola_2014b} corresponds to any strictly positive linear canonical transformation (Proposition \ref{prop_supersymmetric_onto}) and therefore gives a geometric composition law (Theorem \ref{thm_shift_composition}) for a large class of integral operators with nondegenerate Gaussian kernels (Theorem \ref{thm_gaussian_kernels}).

A good example to keep in mind is the shifted harmonic oscillator, considered in \cite[Sec.~VII.D]{Krejcirik_Siegl_Tater_Viola_2014} or \cite{Mityagin_Siegl_Viola_2013}. We write $D_x = \frac{1}{\ii}\partial_x$ and $\Op^w$ for the Weyl quantization. For $b \in \Bbb{R}$, let
\[
	\begin{aligned}
	P_b &= \frac{1}{2}(D_x^2 + x^2 - 2\ii b x - b^2 - 1)
	\\ &= \frac{1}{2}\Op^w(\xi^2 + (x-\ii b)^2 - 1).
	\end{aligned}
\]
In particular, the evolution $\ee^{-\ii (t_1 + \ii t_2)P_b}$ is a bounded operator on $L^2(\Bbb{R})$ if $t_2 < 0$ and $t_1 \in \Bbb{R}$, since the multiplication operator $-2 b x$ is subordinated to the harmonic oscillator $\Re P_b = \frac{1}{2}(x^2 + D_x^2)-1-b^2$. Boundedness for $t_2 < 0$ may also be shown using absolute convergence of the eigenfunction expansion for the evolution \cite{Mityagin_Siegl_Viola_2013}. This subordination argument fails as $t_2 \to 0^-$, and making this precise, in Example \ref{ex_shifted_HO} we compute the norm
\begin{equation}\label{eq_intro_norm_SHO}
	\|\ee^{-\ii(t_1 + \ii t_2)P_b}\|_{\mathcal{L}(L^2(\Bbb{R}))} = \exp\left(\frac{\cos t_1 - \cosh t_2}{\sinh t_2} b^2\right), \quad t_1 \in \Bbb{R}, ~~ t_2 < 0,
\end{equation}
which blows up exponentially rapidly in $1/t_2$ as $t_2 \to 0^-$ if and only if $t_1 \notin 2\pi\Bbb{Z}$. Perhaps more interestingly, we also describe how the norm is a simple consequence of the dynamics on phase space induced by the evolution operator.

The hypotheses used in this paper are satisfied if the quadratic part of the Hamiltonian has negative definite imaginary part, so the reader could skip ahead and substitute this weaker hypothesis in Theorems \ref{thm_norm_quadratic} and \ref{thm_norm_deg2}.

\subsection{Definitions}

To begin, we recall the Weyl quantization; see for instance \cite[Sec.~18.5]{Hormander_ALPDO_3}. 

\begin{definition}\label{def_Weyl}
For a symbol $a \in \mathscr{S}'(\Bbb{R}^{2n})$, the Weyl quantization may be defined weakly for $u, v \in \mathscr{S}(\Bbb{R}^{2n})$ via the formula
\[
	\langle a^w(x,D_x) u, v \rangle = (2\pi)^{-n}\int \ee^{\ii (x-y)\cdot \xi}a\left(\frac{x+y}{2}, \xi\right)u(y)\,\dd y \,\dd \xi.
\]
\end{definition}

We often consider $a(x,\xi)$ a polynomial. In this case, one may obtain $a^w(x,D_x)$ by expanding $a((x+y)/2, \xi)$ and using the rule $x^\alpha \xi^\beta y^\gamma \mapsto x^\alpha D_x^\beta x^\gamma$, where $D_x = -\ii \partial_x$. To find the Weyl quantization of a degree-2 polynomial, the only time where we need to pay attention to the order is in the relation $x_j \xi_j \mapsto \frac{1}{2}(x_j D_{x_j} + D_{x_j}x_j)$.

The Schr\"odinger evolution operators considered are closely linked with complex symplectic linear algebra, made evident in \eqref{eq_Weyl_semigroup} below. Recall the symplectic 2-form on $\Bbb{C}^{2n}$,
\[
	\sigma((x_1, \xi_1), (x_2, \xi_2)) = \xi_1\cdot x_2 - \xi_2 \cdot x_1, \quad (x_1, \xi_1), (x_2, \xi_2) \in \Bbb{C}^{2n}.
\]
A transformation $\Bff{K}$ is canonical if it preserves the symplectic form, $\Bff{K}^*\sigma = \sigma$. We will reserve bold upper-case letters for canonical transformations and bold lower-case letters, such as $\Bff{z} = (x,\xi)$, for vectors in the symplectic vector space $\Bbb{C}^{2n}$. 

Recall also the Hamilton vector field, which may be regarded as a matrix when the Hamiltonian is quadratic:
\[
	\begin{aligned}
	H_q &= \partial_\xi q \partial_x - \partial_x q \partial_\xi
	\\ &= \left(\begin{array}{cc} q''_{\xi x} & q''_{\xi \xi} \\ -q''_{xx} & -q''_{x\xi}\end{array}\right).
	\end{aligned}
\]
The Hamilton flow $\exp H_q$ is always a canonical transformation. We make the following (strict) positivity assumption of Melin and Sj\"ostrand \cite{Melin_Sjostrand_1976} on the Hamilton flow of the quadratic part of our Hamiltonians.

\begin{definition}\label{def_positive_canonical}
A linear canonical transformation $\Bff{K}:\Bbb{C}^{2n} \to \Bbb{C}^{2n}$ is positive if the Hermitian form $\ii \sigma(\overline{\Bff{z}}, \Bff{z})$ increases upon applying $\Bff{K}$ for all $\Bff{z} = (x,\xi) \in \Bbb{C}^{2n}$. Equivalently,
\[
	\ii \left(\sigma(\overline{\Bff{K}\Bff{z}}, \Bff{K}\Bff{z}) - \sigma(\overline{\Bff{z}}, \Bff{z})\right) \geq 0, \quad \forall \Bff{z} \in \Bbb{C}^{2n}.
\]
The transformation $\Bff{K}$ is strictly positive if the inequality is strict for all $\Bff{z} \neq 0$.
\end{definition}

Our positivity assumption is on the flow $\exp H_q$ instead of on the generator $\ii q(x,\xi)$, which is why we say that the evolution is elliptic. We see in Section \ref{ssec_supersymmetric_onto} that strict positivity is a necessary and sufficient condition for defining $\ee^{-\ii Q}$ as a compact operator using the methods of \cite{Aleman_Viola_2014b}, summarized in Section \ref{ssec_Fock}. In particular, the analysis there relies on a hypothesis of supersymmetric structure (Definition \ref{def_supersymmetric}) which always holds when $\exp H_q$ is strictly positive (Proposition \ref{prop_positivity_implies_supersymmetry}).

We emphasize that, in defining the compact operator $\ee^{-\ii P}$ throughout, we do not assume that $\{\ee^{-\ii t P}\}_{t \in [0, 1]}$, defined in the sense of \cite{Aleman_Viola_2014b}, is a family of bounded operators.

\subsection{Results}

With these definitions in hand, we study $\ee^{-\ii P}$ acting on $L^2(\Bbb{R}^n)$ for $P = p^w$ where $p(x,\xi)$ is a degree-two polynomial and the flow of the quadratic part of $p$ is strictly positive. Because this positivity assumption implies that the gradient of the quadratic part is invertible, it suffices to study $p(x,\xi) = q((x,\xi) - \Bff{v})$ for $\Bff{v} \in \Bbb{C}^{2n}$ fixed.

First, in the quadratic case, the $L^2$ operator norm of $\ee^{-\ii Q}$ may be computed from the associated Hamilton flow $\exp H_q$. This result is inspired by \cite[Thm.~4.3]{Hormander_1983}.

\begin{theorem}\label{thm_norm_quadratic}
Let $q:\Bbb{R}^{2n}\to\Bbb{C}$ be a quadratic form for which $\Bff{K} = \exp H_q$ is strictly positive, let $Q = q^w(x,D_x)$, and let $\ee^{-\ii Q}$ be defined as in Section \ref{ssec_Fock}.

Then we may write $\opnm{Spec}\overline{\Bff{K}}^{-1}\Bff{K} = \{\mu_j, \mu_j^{-1}\}_{j=1}^n$, where $\mu_j \in (0,1)$ are repeated for algebraic multiplicity, and
\[
	\|\ee^{-\ii Q}\|_{\mathcal{L}(L^2(\Bbb{R}^n))} = \prod_{j=1}^n \mu_j^{1/4}.
\]
\end{theorem}

This theorem is a straightforward consequence of the beautiful exact classical-quantum correspondence, valid for $Q_j = q_j^w, j=1,2,3,$ with $q_j$ certain complex-valued quadratic forms:
\begin{equation}\label{eq_Weyl_semigroup}
	\exp H_{q_1}\exp H_{q_2} = \exp H_{q_3} \iff \ee^{-\ii Q_1}\ee^{-\ii Q_2} = \pm \ee^{-\ii Q_3}.
\end{equation}
The idea of the proof of Theorem \ref{thm_norm_quadratic}, given in full in Section \ref{ssec_norm_quadratic}, is simple: the operator $(\ee^{-\ii Q})^*\ee^{-\ii Q}$ is associated with the canonical transformation $\overline{\Bff{K}}^{-1}\Bff{K}$. This can be shown to be the Hamilton flow of a quadratic form $q_1$, and since $\overline{\Bff{K}}^{-1}\Bff{K}$ corresponds to a positive definite compact operator, we may take $\ii q_1$ real positive definite. Writing $Q_1 = q_1^w$, by \eqref{eq_Weyl_semigroup}, $\ee^{-\ii Q_1} = (\ee^{-\ii Q})^*\ee^{-\ii Q}$. The spectrum of $\exp H_{q_1} = \overline{\Bff{K}}^{-1}\Bff{K}$ gives $\opnm{Spec} Q_1$ and therefore $\|\ee^{-\ii Q_1}\| = \|\ee^{-\ii Q}\|^2$.

In Proposition \ref{prop_classical_quantum}, we establish \eqref{eq_Weyl_semigroup} when the flows $\exp H_{q_j}$ are strictly positive, though many forms of the correspondence are well-known (see for instance \cite[Prop.~5.9]{Hormander_1995}). We extend this relation to polynomials of degree 2 in Theorem \ref{thm_shift_composition} by keeping track of a constant factor associated with phase-space shifts, and this extension allows us to find the norm of any complex shift of a quadratic Hamiltonian with strictly positive Hamilton flow.

\begin{theorem}\label{thm_norm_deg2}
Let $q:\Bbb{R}^{2n}\to\Bbb{C}$ be a quadratic form for which $\Bff{K} = \exp H_q$ is strictly positive, and let the matrix $A = A(\Bff{K})$ be as in \eqref{eq_def_A_matrix}. For $\Bff{v} \in \Bbb{C}^{2n}$, let $p(x,\xi) = q((x,\xi) - \Bff{v})$ and let $P = p^w$ and $Q = q^w$.

Then, for $\ee^{-\ii Q}$ and $\ee^{-\ii P}$ defined as in Section \ref{ssec_Fock},
\begin{equation}\label{eq_deg2_norm_result}
	\|\ee^{-\ii P}\|_{\mathcal{L}(L^2(\Bbb{R}^n))} = \ee^{-\frac{1}{2}\sigma(\Im \Bff{v}, A\Im \Bff{v})} \|\ee^{-\ii Q}\|_{\mathcal{L}(L^2(\Bbb{R}^n))}.
\end{equation}
\end{theorem}

The method of proof of Theorems \ref{thm_norm_quadratic} and \ref{thm_norm_deg2} gives more information than the operator norm. In fact, one has a reduction to an operator of harmonic oscillator type, similar to \cite[Thm.~2.1]{Aleman_Viola_2014a}. As described in Theorem \ref{thm_SVD}, one may find $\ii Q_2$ of positive harmonic oscillator type, displacements $\Bff{a}_1, \Bff{a}_2 \in \Bbb{R}^{2n}$, unitary metaplectic operators $\mathcal{U}_1, \mathcal{U}_2$, and unitary phase-space shifts $\mathcal{S}_{\Bff{a}_1}, \mathcal{S}_{\Bff{a}_2}$ such that
\[
	\begin{aligned}
	\ee^{-\ii Q} &= \mathcal{U}_2 \ee^{-\frac{\ii}{2} Q_2} \mathcal{U}_1^*,
	\\ \ee^{-\ii P} &= \ee^{\frac{\ii}{2}\sigma(\Bff{v}, \Bff{a}_2 - \Bff{a}_1)}\mathcal{S}_{\Bff{a}_2} \ee^{-\ii Q}\mathcal{S}_{\Bff{a}_1}^*.
	\end{aligned}
\]
Theorems \ref{thm_norm_quadratic} and \ref{thm_norm_deg2} are straightforward corollaries of these decompositions, in particular because $\Bff{a}_2 - \Bff{a}_1 = A\Im \Bff{v}$.

Finally, we note that the geometric meaning associated with Schr\"odinger evolutions may be applied to a broad class of integral operators with Gaussian kernels. By the Mehler formula \cite{Hormander_1995} (see Proposition \ref{prop_Mehler_via_FBI}) and Proposition \ref{prop_Mehler_integrability}, for every $P = \Op^w(q((x,\xi) - \Bff{v}))$ where $q$ is a quadratic form with strictly positive Hamilton flow, there exists some degree-2 polynomial $\varphi(x,y)$ with $\Im \varphi$ positive definite and $\det \varphi''_{xy} \neq 0$ such that $\ee^{-\ii P} = \mathfrak{T}_\varphi$ with
\begin{equation}\label{eq_T_phi}
	\mathfrak{T}_\varphi u(x) = \int \ee^{\ii \varphi(x,y)}u(y)\,dy.
\end{equation}
This association may be reversed.

\begin{theorem}\label{thm_gaussian_kernels}
Let $\varphi:\Bbb{R}^n_x \times \Bbb{R}^n_y \to \Bbb{C}$ be a degree-2 polynomial such that $\Im \varphi''$ is a positive definite $2n\times 2n$ matrix and $\det \varphi''_{xy} \neq 0$. Then there exists a quadratic form $q$ with $\exp H_q$ strictly positive, a vector $\Bff{v} \in \Bbb{C}^{2n}$, and a complex number $c \in \Bbb{C}$ such that, writing $P = \Op^w(q((x,\xi) - \Bff{v})),$ defining $\ee^{-\ii P}$ as in Section \ref{ssec_Fock}, and with $\mathfrak{T}_\varphi$ in \eqref{eq_T_phi},
\[
	\mathfrak{T}_\varphi = c\ee^{-\ii P}.
\]
\end{theorem}

\subsection{Context and plan of the paper}

The primary motivation for this paper is the study of linear perturbations of quadratic operators, particularly to study subelliptic operators as in Examples \ref{ex_Davies} and \ref{ex_shifted_Davies}. The composition formula and the exact formula for the norm draw a sharp contrast with the quadratic case, where the Hamilton flow is enough to completely describe the Schr\"odinger evolution up to sign. Theorems \ref{thm_norm_deg2} and \ref{thm_SVD} show that the complex Hamilton flow nonetheless gives precise information for the evolution, both in terms of the norm and in terms of the dynamics on phase space. The formula is furthermore straightforward to compute and is not limited by an Ehrenfest time or other error term.

Following the ideas of \cite{Hormander_1995} in the quadratic case, we also show that the study of Schr\"odinger evolution operators, as considered in \cite{Aleman_Viola_2014b}, coincides to a large extent with the study of nondegenerate Gaussian kernels studied in, for instance, \cite{Howe_1988, Lieb_1990}. We note, however, that we are only considering the problem on $L^2(\Bbb{R}^n)$, where the fact that the functions maximizing the norm are Gaussians becomes evident from the reduction to a harmonic oscillator model. Here, we focus on what the norm is, where that norm is attained in phase space, and how these objects can be found using elementary symplectic linear algebra.

The association between positive linear canonical transformations and Gaussian kernels is already detailed in \cite{Hormander_1995}, particularly in Theorem 5.12. There, H\"ormander studies this association as an extension from the metaplectic semigroup, defined as the set of Schr\"odinger evolutions of quadratic Hamiltonians with negative semidefinite imaginary parts. The principal novelty of this work is therefore in the extension to polynomials of degree 2, but we also find here that the analysis in \cite{Aleman_Viola_2014b}, inspired by the works of Sj\"ostrand, allows us to directly associate strictly positive canonical transformations to Schr\"odinger evolutions by using a broader definition of the Schr\"odinger evolution. Furthermore, the reduction via an FBI--Bargmann transform leads to simple proofs of established results such as Mehler formulas.

We postpone the limiting case of non-strictly positive canonical transformations for future investigation. For quadratic Hamiltonians (or, equivalently, Gaussian kernels), this limiting case is well-studied, \cite{Hormander_1983, Howe_1988, Hormander_1995}. When considering linear perturbations, strict positivity of the flow assures us the existence of unique maximizers for norms, and allows us to side-step questions of domains for less strongly regularizing operators. This is particularly convenient since we work with shift operators \eqref{eq_def_shift} for complex translations, which are not even defined on arbitrary functions in $\mathscr{S}(\Bbb{R}^n)$. Nonetheless, there is evidence that one could carry the analysis even beyond the set of positive canonical transformations: in Section \ref{ssec_Bargmann_via_Mehler} we observe that the classical Bargmann transform can be formally obtained as a Schr\"odinger evolution $\ee^{-\ii P}$ of an operator $P = p^w$ whose spectrum is $\ii \Bbb{R}$.

The plan of this paper is as follows. In Section \ref{sec_shifts} we introduce phase-space shift operators and prove some associated properties. In Section \ref{sec_proofs_norms}, we prove Theorems \ref{thm_norm_quadratic}, \ref{thm_norm_deg2}, and \ref{thm_SVD}. Section \ref{sec_generators_canonical} serves to collect, re-prove, and extend certain results on the evolution operators considered here, including Mehler formulas and Theorem \ref{thm_gaussian_kernels}. Finally, in Section \ref{sec_app}, we apply Theorems \ref{thm_norm_quadratic}, \ref{thm_norm_deg2}, and \ref{thm_SVD} to some simple concrete models.

\begin{acknowledgements}
The author gratefully acknowledges the support of a d\'el\'egation Centre National de la Recherche Scientifique (CNRS) during the preparation of this manuscript.
\end{acknowledgements}

\section{Shift operators}\label{sec_shifts}

The fundamental tool in introducing linear perturbations will be the phase-space shift operators
\begin{equation}\label{eq_def_shift}
	\mathcal{S}_{(v_x,v_\xi)}u(x) = \ee^{\ii v_\xi\cdot x - \frac{\ii}{2}v_x \cdot v_\xi}u(x-v_x).
\end{equation}
These naturally appear in the proof of Theorem \ref{thm_norm_deg2}, because from Lemma \ref{lem_shift_Egorov}, $\ee^{-\ii P} = \mathcal{S}_{\Bff{v}}\ee^{-\ii Q}\mathcal{S}_{\Bff{v}}^{-1}$.

When $\Bff{v}$ is real, 
\[
	\mathcal{S}_{\Bff{v}} = \Op^w(\ee^{-\ii \sigma(\Bff{z}, \Bff{v})}), \quad \Bff{z} = (x,\xi) \in \Bbb{R}^{2n},
\]
and these operators play a fundamental role in the realization of the Weyl quantization via Fourier decomposition of symbols. When $\Bff{v}$ is not real, $\mathcal{S}_{\Bff{v}}$ is not a bounded operator, even from $\mathscr{S}(\Bbb{R}^n)$ to $\mathscr{S}'(\Bbb{R}^n)$. Nonetheless, we work with operators defined on cores generated by products of polynomials and rapidly decaying Gaussians, as in \cite[Thm.~1.1]{Aleman_Viola_2014b}. The Weyl symbols of these operators are likewise superexponentially decaying in any tubular neighborhood of $\Bbb{R}^n$ along with all their derivatives, shown in Propositions \ref{prop_Mehler_integrability} and \ref{prop_Mehler_via_FBI}. It is on these smooth and rapidly decaying functions and symbols that we perform our computations.

It is straightforward to check that
\begin{equation}\label{eq_shift_composition}
	\mathcal{S}_{\Bff{v}_1}\mathcal{S}_{\Bff{v}_2} = \ee^{\frac{\ii}{2} \sigma(\Bff{v}_1, \Bff{v}_2)}S_{\Bff{v}_1+\Bff{v}_2} = \ee^{\ii \sigma(\Bff{v}_1, \Bff{v}_2)}\mathcal{S}_{\Bff{v}_2}\mathcal{S}_{\Bff{v}_1},
\end{equation}
from which it is obvious that \eqref{eq_Weyl_semigroup} cannot be extended to symbols of degree 2. We also record that
\begin{equation}\label{eq_shift_inverse}
	\mathcal{S}_{\Bff{v}}^{-1} = \mathcal{S}_{-\Bff{v}}
\end{equation}
and
\begin{equation}\label{eq_shift_adjoint}
	\mathcal{S}_{\Bff{v}}^* = \mathcal{S}_{-\overline{\Bff{v}}}.
\end{equation}

For an appropriate (very rapidly decaying) symbol $p:\Bbb{R}^{2n}\to \Bbb{C}$, or for $p:\Bbb{R}^{2n}\to \Bbb{C}$ a polynomial and $p^w$ acting on rapidly decaying functions, it is straightforward to check the following Egorov relations. The two-sided relation \eqref{eq_shift_Egorov} is generally interpreted to mean that $\mathcal{S}_{\Bff{z}_0}$ ``quantizes'' the canonical transformation $\Bff{z} \mapsto \Bff{z} + \Bff{v}$ just as $\ee^{-\ii Q}$ quantizes $\Bff{K}$ in \eqref{eq_semigroup_Egorov}. The simple half-Egorov relations \eqref{eq_shift_Egorov_left} and \eqref{eq_shift_Egorov_right}, however, are quite special and crucial for the analysis which follows. Since we apply this lemma to integrable Gaussian symbols, we make no effort to find an optimal symbol class.

\begin{lemma}\label{lem_shift_Egorov}
Let $a:\Bbb{C}^{2n}\to \Bbb{C}$ be either a polynomial or smooth and rapidly decaying on tubular neighborhoods of $\Bbb{R}^{2n}$ in that, for every $C > 0$, there exists $c > 0$ for which $\ee^{c|\Bff{z}|^2}\partial^\alpha a(\Bff{z}) \in L^\infty(\{|\Im \Bff{z}| < C\})$ for any multi-index $\alpha \in \Bbb{N}^{2n}$. Then
\begin{equation}\label{eq_shift_Egorov_left}
	\mathcal{S}_{\Bff{v}}a^w(x, D_x) = \Op^w\left(a(\Bff{z} - \frac{1}{2}\Bff{v})\ee^{-\ii \sigma(\Bff{z}, \Bff{v})}\right),
\end{equation}
\begin{equation}\label{eq_shift_Egorov_right}
	a^w(x, D_x)\mathcal{S}_{\Bff{v}}^{-1} = \Op^w\left(a(\Bff{z} - \frac{1}{2}\Bff{v})\ee^{\ii \sigma(\Bff{z}, \Bff{v})}\right),
\end{equation}
and
\begin{equation}\label{eq_shift_Egorov}
	\mathcal{S}_{\Bff{v}}a^w(x, D_x) = \Op^w(a(\Bff{z}-\Bff{v}))\mathcal{S}_{\Bff{v}}.
\end{equation}
\end{lemma}

\begin{proof}
By density in $L^2(\Bbb{R}^n)$, it suffices to check the identity for $u$ holomorphic and sufficiently rapidly decaying, for instance, any polynomial times a Gaussian.

We write
\[
	\mathcal{S}_{\Bff{v}}a^w u(x) = (2\pi)^{-n}\int \ee^{\ii(x-v_x-y)\xi + \ii v_\xi x - \frac{\ii}{2}v_x v_\xi}a\left(\frac{x+y}{2}-\frac{v_x}{2}, \xi\right)u(y)\,\dd y \,\dd \xi
\]
and rearrange the phase:
\[
	\ii(x-v_x-y)\xi + \ii v_x x - \frac{\ii}{2} v_x v_\xi =  \ii (x-y)(\xi - \frac{v_\xi}{2}) + \ii v_\xi\frac{x+y}{2}-\ii v_x(\xi + \frac{v_\xi}{2}).
\]
Changing variables in $\xi$ gives \eqref{eq_shift_Egorov_left}. A similar argument, or computing the adjoint $(a^w \mathcal{S}_{\Bff{v}}^{-1})^* = \mathcal{S}_{\overline{\Bff{v}}}\overline{a}^w$, gives \eqref{eq_shift_Egorov_right}. The two together then give \eqref{eq_shift_Egorov}.
\end{proof}

The key to the proof of Theorems \ref{thm_norm_deg2} and \ref{thm_shift_composition} is the observation that, for the Schr\"odinger evolution of a quadratic operator, a one-sided shift is equivalent to a certain two-sided shift up to a geometric factor.

\begin{proposition}\label{prop_crossing}
	Let $q:\Bbb{R}^{2n}\to\Bbb{C}$ be a quadratic form for which $\Bff{K} = \exp H_q$ is strictly positive (Definition \ref{def_positive_canonical}). Let $Q = q^w$ and let $\ee^{-\ii Q}$ be defined as in Section \ref{ssec_Fock}. Fix $\Bff{v}\in\Bbb{C}^{2n}$ and recall the definition \eqref{eq_def_shift} of the phase-space shift operators.
	
	Then, if 
	\[
		\begin{aligned}
		\Bff{u} &= (1-\Bff{K}^{-1})\Bff{v},
		\\ \Bff{w} &= (1-\Bff{K})\Bff{v},
		\end{aligned}
	\]
	then 
	\[
		\mathcal{S}_{\Bff{v}}\ee^{-\ii Q}\mathcal{S}_{\Bff{v}}^{-1} = \ee^{\frac{\ii}{2}\sigma(\Bff{u}, \Bff{v})}\ee^{-\ii Q}\mathcal{S}_{\Bff{u}}^{-1} = \ee^{\frac{\ii}{2}\sigma(\Bff{v}, \Bff{w})}\mathcal{S}_{\Bff{w}}\ee^{-\ii Q}.
	\]
\end{proposition}

\begin{proof} We prove the first equality, and the second follows similarly.

Let
\[
	T = \tanh(H_q/2) = (\Bff{K}+1)^{-1}(\Bff{K}-1),
\]
recalling that strict positivity of $\Bff{K}$ implies that $-1 \notin \opnm{Spec}\Bff{K}$. Disregarding constants in the Mehler formula \eqref{eq_Mehler}, it suffices to verify the equality for the symbol
\[
	a(\Bff{z}) = \exp\left(\sigma(\Bff{z}, \frac{1}{\ii}T\Bff{z})\right), \quad \Bff{z} = (x,\xi) \in \Bbb{C}^{2n}.
\]
By Lemma \ref{lem_shift_Egorov},
\[
	 \mathcal{S}_{\Bff{v}} a^w \mathcal{S}_{\Bff{v}}^{-1} = \ee^{c_0}a^w \mathcal{S}_{\Bff{u}}^{-1}
\]
if and only if
\[
	\sigma(\Bff{z}-\Bff{v}, \frac{1}{\ii}T(\Bff{z}-\Bff{v})) = c_0 + \sigma(\Bff{z} - \Bff{u}/2, \frac{1}{\ii} T (\Bff{z} - \Bff{u}/2)) + \ii \sigma(\Bff{z}, \Bff{u}).
\]
Eliminating the term $\sigma(\Bff{z}, \frac{1}{\ii}T\Bff{z})$ from both sides and using antisymmetry of $T$ with respect to $\sigma$, we obtain the equivalent
\[
	2\sigma(\frac{1}{\ii} T\Bff{v}, \Bff{z}) + \sigma(\Bff{v}, \frac{1}{\ii}T\Bff{v}) = c_0 + \sigma(\frac{1}{\ii}(1+T)\Bff{u}, \Bff{z}) + \frac{1}{4}\sigma(\Bff{u}, \frac{1}{\ii}T\Bff{u}).
\]
From this, we deduce that
\[
	\begin{aligned}
	\Bff{u} &= 2(1+T)^{-1}T\Bff{v}
	\\ &= 2\left((\Bff{K}+1)2\Bff{K}^{-1}\right)^{-1}(\Bff{K}+1)^{-1}(\Bff{K}-1)^{-1}\Bff{v}
	\\ &= (1-\Bff{K}^{-1})\Bff{v}.
	\end{aligned}
\]

To compute the constant, we use that $\Bff{v} = \frac{1}{2}(T^{-1}+1)\Bff{u}$ to obtain
\[
	\begin{aligned}
	c_0 &= \sigma(\Bff{v}, \frac{1}{\ii}T\Bff{v}) - \frac{1}{4}\sigma(\Bff{u}, \frac{1}{\ii}T\Bff{u})
	\\ &= \frac{1}{4\ii}\sigma((T^{-1}+1)\Bff{u}, (1+T)\Bff{u}) - \frac{1}{4\ii}\sigma(\Bff{u}, T\Bff{u})
	\\ &=  \frac{1}{4\ii}\left(\sigma(T^{-1}\Bff{u}, \Bff{u}) + \sigma(T^{-1}\Bff{u}, T\Bff{u}) + \sigma(\Bff{u}, \Bff{u}) + \sigma(\Bff{u}, T\Bff{u}) - \sigma(\Bff{u}, T\Bff{u})\right)
	\\ &= \frac{1}{4\ii}\sigma(T^{-1}\Bff{u}, \Bff{u}),
	\end{aligned}
\]
since $T$ is antisymmetric with respect to $\sigma$ and $\sigma(\Bff{u}, \Bff{u}) = 0$. Since
\[
	\sigma(T^{-1}\Bff{u}, \Bff{u}) = \sigma((1+T^{-1})\Bff{u}, \Bff{u}) = \sigma(2\Bff{v}, \Bff{u}),
\]
we obtain
\[
	c_0 = \frac{\ii}{2}\sigma(\Bff{u}, \Bff{v}),
\]
proving the first equality in the proposition. Again, the second equality follows by a similar argument.
\end{proof}

As a consequence, we are able to extend the composition relation \eqref{eq_Weyl_semigroup} to shifted operators by including a geometric coefficient.

\begin{theorem}\label{thm_shift_composition}
For $j = 1, 2$, let $\Bff{v}_j \in \Bbb{C}^{2n}$, let $q_j$ be quadratic forms for which $\Bff{K}_j = \exp H_{q_j}$ are strictly positive. Let
\[
	P_j = \Op^w(q_j((x,\xi) - \Bff{v}_j))
\]
and let $\ee^{-\ii P_j}$ be defined as in Section \ref{ssec_Fock}. By Proposition \ref{prop_supersymmetric_onto}, let $q_3$ be a quadratic form such that, with $Q_3 = q_3^w$,
\begin{equation}\label{eq_shift_composition_quadratic}
	\ee^{-\ii Q_3} = \pm \ee^{-\ii Q_1} \ee^{-\ii Q_2},
\end{equation}
and let $\Bff{K}_3 = \exp H_{q_3} = \Bff{K}_1\Bff{K}_2$.

As in Proposition \ref{prop_crossing}, let
\[
	\begin{aligned}
	\Bff{w}_1 &= (1-\Bff{K}_1)\Bff{v}_1,
	\\ \Bff{u}_2 &= (1-\Bff{K}_2^{-1})\Bff{v}_2,
	\end{aligned}
\]
and define
\begin{equation}\label{eq_def_v3}
	\Bff{v}_3 = (1-\Bff{K}_3)^{-1} \Bff{w}_1 + (1-\Bff{K}_3^{-1})^{-1}\Bff{u}_2.
\end{equation}
Then, with
\[
	P_3 = \Op^w(q_3((x,\xi) - \Bff{v}_3) = \mathcal{S}_{\Bff{v}_3} Q_3\mathcal{S}_{\Bff{v}_3}^{-1},
\]
and the same sign as in \eqref{eq_shift_composition_quadratic}, one has
\[
	\ee^{-\ii P_1}\ee^{-\ii P_2} = \pm \ee^{\frac{\ii}{2}(\sigma(\Bff{v}_1 - \Bff{v}_3, \Bff{w}_1) + \sigma(\Bff{u}_2, \Bff{v}_2 - \Bff{v}_3))} \ee^{-\ii P_3}.
\]
\end{theorem}

\begin{proof}
We begin by using Proposition \ref{prop_crossing} to push shift operators to the outside of the composition:
\[
	\begin{aligned}
	\ee^{-\ii P_1} \ee^{-\ii P_2} &= \mathcal{S}_{\Bff{v}_1} \ee^{-\ii Q_1}\mathcal{S}_{\Bff{v}_1}^{-1}\mathcal{S}_{\Bff{v}_2} \ee^{-\ii Q_2}\mathcal{S}_{\Bff{v}_2}^{-1}
	\\ &= \ee^{\frac{\ii}{2}\sigma(\Bff{v}_1, \Bff{w}_1) + \frac{\ii}{2}\sigma(\Bff{u}_2, \Bff{v}_2)} \mathcal{S}_{\Bff{w}_1}\ee^{-\ii Q_1}\ee^{-\ii Q_2}\mathcal{S}_{\Bff{u}_2}^{-1}
	\\ &= \pm \ee^{\frac{\ii}{2}\sigma(\Bff{v}_1, \Bff{w}_1) + \frac{\ii}{2}\sigma(\Bff{u}_2, \Bff{v}_2)} \mathcal{S}_{\Bff{w}_1}\ee^{-\ii Q_3}\mathcal{S}_{\Bff{u}_2}^{-1}.
	\end{aligned}
\]
Throughout this proof, the symbol $\pm$ makes reference to the sign in \eqref{eq_shift_composition_quadratic}. We continue by pushing the $\Bff{w}_1$ shift to a two-sided shift by
\[
	\Bff{v}_{3a} = (1-\Bff{K}_3)^{-1}\Bff{w}_1.
\]
This gives
\[
	\mathcal{S}_{\Bff{w}_1}\ee^{-\ii Q_3}\mathcal{S}_{\Bff{u}_2}^{-1} = \ee^{-\frac{\ii}{2}\sigma(\Bff{v}_{3a}, \Bff{w}_1)} \mathcal{S}_{\Bff{v}_{3a}} \ee^{-\ii Q_3}\mathcal{S}_{\Bff{v}_{3a}}^{-1}\mathcal{S}_{\Bff{u}_2}^{-1}.
\]
Recalling the composition formula \eqref{eq_shift_composition} and pushing $\Bff{u}_2$ to a two-sided shift by
\[
	\Bff{v}_{3b} = (1-\Bff{K}_3^{-1})^{-1}\Bff{u}_2
\]
gives
\[
	\begin{aligned}
	\mathcal{S}_{\Bff{v}_{3a}} \ee^{-\ii Q_3}\mathcal{S}_{\Bff{v}_{3a}}^{-1}\mathcal{S}_{\Bff{u}_2}^{-1} &= \ee^{-\frac{\ii}{2}\sigma(\Bff{u}_2, \Bff{v}_{3a})}\mathcal{S}_{\Bff{v}_{3a}} \ee^{-\ii Q_3} \mathcal{S}_{\Bff{u}_2}^{-1}\mathcal{S}_{\Bff{v}_{3a}}^{-1}
	\\ &= \ee^{-\ii\sigma(\Bff{u}_2, \Bff{v}_{3a}) - \frac{\ii}{2}\sigma(\Bff{u}_2, \Bff{v}_{3b})} \mathcal{S}_{\Bff{v}_{3a}} \mathcal{S}_{\Bff{v}_{3b}} \ee^{-\ii Q_3} \mathcal{S}_{\Bff{v}_{3b}}^{-1}\mathcal{S}_{\Bff{v}_{3a}}^{-1}
	\end{aligned}
\]

Finally, we set
\[
	\Bff{v}_3 = \Bff{v}_{3a} + \Bff{v}_{3b},
\]
giving \eqref{eq_def_v3}, and we note that $\mathcal{S}_{\Bff{v}_{3a}}\mathcal{S}_{\Bff{v}_{3b}} = c\mathcal{S}_{\Bff{v}_3}$ while $\mathcal{S}_{\Bff{v}_{3b}}^{-1}\mathcal{S}_{\Bff{v}_{3a}}^{-1} = c^{-1}\mathcal{S}_{\Bff{v}_3}^{-1}$.
Therefore
\[
	\mathcal{S}_{\Bff{v}_{3a}} \mathcal{S}_{\Bff{v}_{3b}} \ee^{-\ii Q_3} \mathcal{S}_{\Bff{v}_{3b}}^{-1}\mathcal{S}_{\Bff{v}_{3a}}^{-1} = \mathcal{S}_{\Bff{v}_3} \ee^{-\ii Q_3} \mathcal{S}_{\Bff{v}_3}^{-1} = \ee^{-\ii P_3}.
\]
Combining these computations, we obtain that
\begin{equation}\label{eq_composition_proof_w_c0}
	\ee^{-\ii P_1}\ee^{-\ii P_2} = \pm \ee^{\frac{\ii}{2}c_0} \ee^{-\ii P_3}
\end{equation}
where
\begin{equation}\label{eq_composition_proof_c0_v1}
	c_0 = \sigma(\Bff{v}_1 - \Bff{v}_{3a}, \Bff{w}_1) - 2\sigma(\Bff{u}_2, \Bff{v}_{3a}) + \sigma(\Bff{u}_2, \Bff{v}_2 - \Bff{v}_{3b}).
\end{equation}

To simplify $c_0$, we are motivated by the idea that the result should be independent of our choice to push $\mathcal{S}_{\Bff{w}_1}$ across $\ee^{-\ii Q_3}$ first and $\Bff{S}_{\Bff{u}_2}^{-1}$ second. If we reverse the order, we obtain instead
\[
	c_0 = \sigma(\Bff{v}_1 - \Bff{v}_{3a}, \Bff{w}_1) - 2\sigma(\Bff{v}_{3b}, \Bff{w}_1) + \sigma(\Bff{u}_2, \Bff{v}_2 - \Bff{v}_{3b}).
\]
To check that
\begin{equation}\label{eq_composition_proof_any_order}
	\sigma(\Bff{u}_2, \Bff{v}_{3a}) = \sigma(\Bff{v}_{3b}, \Bff{w}_1),
\end{equation}
we rewrite the statement as
\begin{equation}\label{eq_composition_proof_any_order_2}
	\sigma(\Bff{u}_2, (1-\Bff{K}_3)^{-1}\Bff{w}_1) = \sigma((1-\Bff{K}_3^{-1})^{-1}\Bff{u}_2, \Bff{w}_1).
\end{equation}
This follows readily by using the transpose with respect to $\sigma$, which may be written as $A^{\sigma \top} = -\Bff{J}A^\top \Bff{J}$ when $\Bff{J}(x,\xi) = (-\xi, x)$. This transpose is linear in $A$, commutes with inverses, and preserves the identity matrix. Furthermore, $\Bff{K}_3$ is canonical, so $\Bff{K}_3^{\sigma\top} = \Bff{K}_3^{-1}$. This allows us to verify that
\[
	\left((1-\Bff{K}_3)^{-1}\right)^{\sigma\top} = (1-\Bff{K}_3^{-1})^{-1}
\]
which gives \eqref{eq_composition_proof_any_order_2} and therefore \eqref{eq_composition_proof_any_order}.

We exploit this observation by writing
\[
	-2\sigma(\Bff{u}_2, \Bff{v}_{3a}) = -\sigma(\Bff{u}_2, \Bff{v}_{3a}) + \sigma(\Bff{v}_{3b}, \Bff{w}).
\]
Putting this into \eqref{eq_composition_proof_c0_v1} and recalling that $\Bff{v}_3 = \Bff{v}_{3a} + \Bff{v}_{3b}$ gives 
\[
	c_0 = \sigma(\Bff{v}_1 - \Bff{v}_3, \Bff{w}_1) + \sigma(\Bff{u}_2, \Bff{v}_2 - \Bff{v}_3).
\]
Along with \eqref{eq_composition_proof_w_c0}, this proves the proposition.
\end{proof}

\section{Proofs of norm results}\label{sec_proofs_norms}

\subsection{The purely quadratic case}\label{ssec_norm_quadratic}

In this section, we prove Theorem \ref{thm_norm_quadratic} and a related singular-value-type decomposition for $\ee^{-\ii Q}$, where $Q = q^w$ for $q$ a quadratic form for which $\Bff{K} = \exp H_q$ is strictly positive.

\begin{proof}
Because $Q$ is defined by the Weyl quantization, $Q^* = \overline{q}^w$. Therefore the operator $(\ee^{-\ii Q})^*\ee^{-\ii Q}$ is associated with the canonical transformation $\overline{\Bff{K}}^{-1}\Bff{K}$.

It is straightforward to check that $\overline{\Bff{K}}^{-1}$ is strictly positive, and therefore $\overline{\Bff{K}}^{-1}\Bff{K}$ is as well. By Proposition \ref{prop_supersymmetric_onto}, there exists some quadratic form $q_1$ such that 
\begin{equation}\label{eq_quad_def_q1}
	\exp H_{q_1} = \overline{\Bff{K}}^{-1}\Bff{K}.
\end{equation}
Let $Q_1 = q_1^w.$ By the classical-quantum correspondence as in Proposition \ref{prop_classical_quantum}, we have
\begin{equation}\label{eq_quad_q_q1}
	\pm \ee^{-\ii Q_1} = (\ee^{-\ii Q})^*\ee^{-\ii Q}.
\end{equation}
The operator on the right is compact and positive definite Hermitian, so by Proposition \ref{prop_real_log}, we may choose $-\ii q_1$ negative definite real. As a consequence, the sign in the equality is $+$.

With this choice, it is classical that there exists some real linear canonical transformation $\Bff{U}$ such that
\begin{equation}\label{eq_quad_def_q2}
	-\ii q_2(x,\xi) := -\ii q_1 \circ \Bff{U}^{-1}(x,\xi) = -\sum_{j=1}^n \frac{\lambda_j}{2}(x_j^2 + \xi_j^2),
\end{equation}
where $\lambda_j > 0$ for $j=1,\dots, n$. There is a metaplectic operator $\mathcal{U}$ quantizing $\Bff{U}^{-1}$, meaning in particular that, with $Q_2 = q_2^w$,
\begin{equation}\label{eq_quad_q1_q2}
	\mathcal{U} \ee^{-\ii Q_1}\mathcal{U}^* = \ee^{-\ii \mathcal{U}Q_1 \mathcal{U}^*} = \ee^{-\ii Q_2}.
\end{equation}

As the direct sum of harmonic oscillators,
\[
	\opnm{Spec} (-\ii Q_2) = \left\{-\sum_{j=1}^n \frac{\lambda_j}{2}(1 + 2\alpha_j)\::\: \alpha \in \Bbb{N}^n\right\},
\]
so
\[
	\|\ee^{-\ii Q_2}\| = \ee^{-\sum \lambda_j/2}.
\]

By \eqref{eq_quad_q_q1} and \eqref{eq_quad_q1_q2},
\[
	\|\ee^{-\ii Q}\| = \|\ee^{-\ii Q_1}\|^{1/2} = \|\ee^{-\ii Q_2}\|^{1/2} = \ee^{-\sum \lambda_j/4}.
\]
The Hamilton vector field of the harmonic oscillator model $q_0(x,\xi) = \frac{1}{2}(\xi^2 + x^2)$ in dimension one is rotation by $-\pi/2$, or $H_{q_0}(x,\xi) = (\xi, -x)$, so $\opnm{Spec}H_{q_0} = \{\pm \ii\}$. Since $-\ii q_2$ is a direct sum of harmonic oscillators,
\[
	\opnm{Spec}H_{-\ii q_2} = \{\pm \ii \lambda_j\}_{j=1}^n.
\]
By the spectral mapping theorem,
\[
	\opnm{Spec}\exp H_{q_2} = \{\ee^{\pm \lambda_j}\}_{j=1}^n.
\]
Writing $\mu_j = \ee^{-\lambda_j}$ gives that $\opnm{Spec} \exp H_{q_2} = \{\mu_j, \mu_j^{-1}\}_{j=1}^n$ and that
\[
	\|\ee^{-\ii Q}\| = \prod \left\{\mu_j^{1/4}\::\: \mu_j \in \opnm{Spec} \exp H_{q_2} \cap (0, 1)\right\}.
\]

The composition \eqref{eq_quad_def_q2} with a linear canonical transformation induces the similarity relation 
\begin{equation}\label{eq_Hq1_Hq2}
	H_{q_2} = \Bff{U} H_{q_1}\Bff{U}^{-1}.
\end{equation}
Therefore, using the definition \eqref{eq_quad_def_q1} of $q_1$,
\[
	\opnm{Spec} \exp H_{q_2} = \opnm{Spec} \exp H_{q_1} = \opnm{Spec} \overline{\Bff{K}}^{-1}\Bff{K},
\]
which completes the proof of the theorem.
\end{proof}

We continue with the proof of the first part of Theorem \ref{thm_SVD}.

\begin{proof}[Proof of \eqref{eq_SVD_quad}]
Using $q_2$ from \eqref{eq_quad_def_q2}, let
\[
	\Bff{L} = \exp(\frac{1}{2}H_{q_2}).
\]
By \eqref{eq_quad_def_q1} and \eqref{eq_Hq1_Hq2}, $\Bff{U}\Bff{L}\Bff{U}^{-1}$ is a natural square root of $\overline{\Bff{K}}^{-1}\Bff{K}.$ Furthermore, because $-\ii q_2$ is real, $\overline{\Bff{L}} = \Bff{L}^{-1}$.

Imitating the singular value decomposition, we look for a metaplectic operator $\mathcal{V}$ quantizing a real linear canonical transformation $\Bff{V}:\Bbb{R}^{2n}\to\Bbb{R}^{2n}$ such that
\begin{equation}\label{eq_quad_SVD}
	\ee^{-\ii Q} = \mathcal{V} \ee^{-\frac{\ii}{2} Q_2} \mathcal{U}^*.
\end{equation}
On the level of canonical transformations, this means that
\[
	\Bff{K} = \Bff{V} \Bff{L} \Bff{U}^{-1},
\]
or
\[
	\Bff{V} = \Bff{K} \Bff{U} \Bff{L}^{-1}.
\]
Therefore, using that $\overline{\Bff{L}} = \Bff{L}^{-1}$, we can prove that $\overline{\tilde{\Bff{U}}}$ is real:
\[
	\begin{aligned}
	\overline{\Bff{V}} &= \overline{\Bff{K}} \Bff{U}\overline{\Bff{L}^{-1}}
	\\ &= \overline{\Bff{K}} \Bff{U}\Bff{L}
	\\ &= \overline{\Bff{K}} \Bff{U}\Bff{L}^2 \Bff{U}^{-1}\Bff{U}\Bff{L}^{-1}
	\\ &= \overline{\Bff{K}\Bff{K}^{-1}}\Bff{KUL}^{-1}
	\\ &= \Bff{V}.
	\end{aligned}
\]

Since $\Bff{V}$ is real, a corresponding metaplectic transformation $\mathcal{V}$ which quantizes $\Bff{V}$ exists. By the classical-quantum correspondence \eqref{eq_Weyl_semigroup} (we have not proven the classical-quantum correspondence here for canonical transformations which are merely positive; we are essentially relying on \cite[Prop.~2.9]{Hormander_1995}), we have that \eqref{eq_quad_SVD} holds up to sign. We may change the sign of either $\mathcal{U}$ or $\mathcal{V}$ freely, since multiplication by $-1$ is an element of the metaplectic group. This gives a singular-value-type decomposition of $\ee^{-\ii Q}$ and therefore proves the first part of Theorem \ref{thm_SVD}.
\end{proof}

\subsection{Proof of norms and decomposition for shifted operators}

We continue by analyzing $P = \Op^w(q(\Bff{z} - \Bff{v}))$. In summary, the canonical transformation associated with $(\ee^{\ii P})^*(\ee^{\ii P})$ allows us to identify the real displacement $\Bff{a}_1$ for which
\[
	(\ee^{\ii P})^*\ee^{\ii P} = cS_{\Bff{a}_1}(\ee^{\ii Q})^* \ee^{\ii Q}S_{\Bff{a}_1}^{-1},
\]
and the constant $c$ is a consequence of Proposition \ref{prop_crossing}.

We remark that one could prove Theorem \ref{thm_norm_deg2} by applying Theorem \ref{thm_shift_composition} to $(\ee^{-\ii P})^*\ee^{-\ii P}$. In order to avoid a somewhat lengthy and redundant computation, we directly prove the second part of Theorem \ref{thm_SVD} from which Theorem \ref{thm_norm_deg2} is an immediate consequence.

\begin{proof}[Proof of Theorem \ref{thm_norm_deg2}]
In this section, we use the notation
\[
	\Bff{K}_1 = \exp H_q
\]
for the canonical transformation associated with $\ee^{-\ii Q}$ and
\[
	\Bff{K}_2 = \overline{\Bff{K}_1}^{-1} = \exp H_{-\overline{q}}
\]
for the canonical transformation associated with $(\ee^{-\ii Q})^* = \ee^{\ii Q^*}$. We will find $\Bff{a}_1, \Bff{a}_2 \in \Bbb{R}^{2n}$ such that, for constants $c_1, c_2, c_3$,
\begin{equation}\label{eq_SVD_1}
	(\ee^{-\ii P})^*\ee^{-\ii P} = c_1\mathcal{S}_{\Bff{a}_1}(\ee^{-\ii Q})^*\ee^{-\ii Q}\mathcal{S}_{\Bff{a}_1}^{-1},
\end{equation}
\begin{equation}\label{eq_SVD_2}
	\ee^{-\ii P}(\ee^{-\ii P})^* = c_2\mathcal{S}_{\Bff{a}_2}\ee^{-\ii Q}(\ee^{-\ii Q})^*\mathcal{S}_{\Bff{a}_2}^{-1},
\end{equation}
and consequently
\begin{equation}\label{eq_SVD_3}
	\ee^{-\ii P} = c_3 \mathcal{S}_{\Bff{a}_2}\ee^{-\ii Q}\mathcal{S}_{\Bff{a}_1}^{-1}.
\end{equation}

Recall from Section \ref{ssec_Fock} that
\[
	\ee^{-\ii P} = \mathcal{S}_{\Bff{v}} \ee^{-\ii Q} \mathcal{S}_{\Bff{v}}^{-1}.
\]
Since $\mathcal{S}_{\Bff{v}}$ is associated with the canonical transformation $\Bff{z} \mapsto \Bff{z} + \Bff{v}$ (see Lemma \ref{lem_shift_Egorov}) and $\ee^{-\ii Q}$ is associated with $\Bff{K}_1 = \exp H_q$ (see Proposition \ref{prop_Mehler_via_FBI}), $\ee^{-\ii P}$ is associated with
\begin{equation}\label{eq_Egorov_exp-iP}
	\Bff{z}\mapsto \Bff{K}_1(\Bff{z} - \Bff{v}) + \Bff{v}.
\end{equation}
Because $\mathcal{S}^*_{\Bff{v}} = \mathcal{S}_{-\overline{\Bff{v}}}$ and $(\ee^{-\ii Q})^* = \ee^{\ii \overline{q}^w}$ is associated with $\Bff{K}_2 = \exp(-H_{\overline{q}}) = \overline{\Bff{K}_1}^{-1}$, we see that $(\ee^{-\ii P})^* \ee^{-\ii P}$ is associated with
\[
	\begin{aligned}
	\Bff{z} &\mapsto \Bff{K}_2(\Bff{K}_1(\Bff{z} - \Bff{v}) + \Bff{v} - \overline{\Bff{v}}) + \overline{\Bff{v}}
	\\ &= \Bff{K}_2\Bff{K}_1 \Bff{z} + (1-\Bff{K}_2\Bff{K}_1)\Re\Bff{v} + \ii(-1+2\Bff{K}_2-\Bff{K}_2\Bff{K}_1)\Im \Bff{v}.
	\end{aligned}
\]
On the other hand, for $\Bff{a}_1 \in \Bbb{C}^{2n}$ to be determined, $\mathcal{S}_{\Bff{a}_1} (\ee^{-\ii Q})^*\ee^{-\ii Q}\mathcal{S}_{\Bff{a}_1}^{-1}$ is associated with the canonical transformation
\[
	\begin{aligned}
	\Bff{z} &\mapsto \Bff{K}_2\Bff{K}_1(\Bff{z} - \Bff{a}_1)+\Bff{a}_1
	\\ &= \Bff{K}_2\Bff{K}_1\Bff{z} + (1-\Bff{K}_2\Bff{K}_1)\Bff{a}_1.
	\end{aligned}
\]

In order to make \eqref{eq_SVD_1} hold on the level of canonical transformations, we set
\begin{equation}\label{eq_a1_no_ImK}
	\Bff{a}_1 = \Re \Bff{v} + \ii(1-\Bff{K}_2\Bff{K}_1)^{-1}(-1 + 2\Bff{K}_2 - \Bff{K}_2\Bff{K}_1)\Im \Bff{v}.
\end{equation}
Notice that $1 \notin \opnm{Spec}\Bff{K}_2\Bff{K}_1$ because $\Bff{K}_2\Bff{K}_1$ is a strictly positive canonical transformation. To check that $\Bff{a}_1$ is real, as it should be since $(\ee^{-\ii P})^*\ee^{-\ii P}$ is self-adjoint, note that
\[
	(1 - \Bff{K}_2\Bff{K}_1)^{-1} = (\Bff{K}_2 - \Bff{K}_1)^{-1} \overline{\Bff{K}_1} = \frac{\ii}{2}(\Im \Bff{K}_1)^{-1}\overline{\Bff{K}_1}.
\]
Therefore
\begin{equation}\label{eq_a1_ImK}
	\begin{aligned}
	\Bff{a}_1 &= \Re \Bff{v} -\frac{1}{2}(\Im \Bff{K}_1)^{-1}(-\overline{\Bff{K}_1} + 2 - \Bff{K}_1)\Im \Bff{v}
	\\ &= \Re \Bff{v} + (\Im \Bff{K}_1)^{-1}(\Re \Bff{K}_1 - 1) \Im \Bff{v}.
	\end{aligned}
\end{equation}

Because the canonical transformation associated with $(\ee^{-\ii P})^*$ is the same as the canonical transformation associated with $\ee^{-\ii P}$ except $\Bff{K}_1$ is replaced by $\Bff{K}_2$ and $\Bff{v}$ is replaced by $\overline{\Bff{v}}$, \eqref{eq_SVD_2} holds on the level of canonical transformations when
\begin{equation}\label{eq_a2_ImK}
	\Bff{a}_2 = \Re \Bff{v} - (\Im \Bff{K}_2)^{-1}(\Re \Bff{K}_2 - 1) \Im \Bff{v}.
\end{equation}
On the other hand, \eqref{eq_SVD_3} holds when
\[
	\Bff{K}_1(\Bff{z} -\Bff{v}) + \Bff{v} = \Bff{K}_1(\Bff{z} - \Bff{a}_1) + \Bff{a}_2,
\]
or
\begin{equation}\label{eq_a1a2_SVD_3}
	\Bff{v} = (1-\Bff{K}_1)^{-1} \Bff{a}_2 + (1-\Bff{K}_1^{-1})^{-1}\Bff{a}_1.
\end{equation}
We will proceed to verify the equivalent statement that $\Bff{a}_2$ from \eqref{eq_a2_ImK} satisifes
\begin{equation}\label{eq_a2a1v}
	\Bff{a}_2 - \Bff{K}_1\Bff{a}_1 = (1-\Bff{K}_1)\Bff{v}
\end{equation}
in assuming that $\Re \Bff{v} = 0$, since the relation is linear and obvious when $\Im \Bff{v} = 0$.

From \eqref{eq_a1_no_ImK}, we obtain that, when $\Re \Bff{v} = 0$,
\[
	\begin{aligned}
	\Bff{K}_1\Bff{a}_1 &= \ii\Bff{K}_1(1-\Bff{K}_2\Bff{K}_1)(-1 + 2\Bff{K}_2 - \Bff{K}_2\Bff{K}_1)\Im \Bff{v}
	\\ &= \ii(\Bff{K}_1^{-1} - \Bff{K}_2)^{-1}(-1 + 2\Bff{K}_2 - \Bff{K}_2\Bff{K}_1)\Im \Bff{v}
	\\ &= \ii(1-\Bff{K}_1\Bff{K}_2)^{-1}(-\Bff{K}_1 + 2\Bff{K}_1\Bff{K}_2 - \Bff{K}_1\Bff{K}_2\Bff{K}_1)\Im \Bff{v}.
	\end{aligned}
\]
For $\Bff{a}_2$ in \eqref{eq_a2_ImK}, the corresponding expression, when $\Re \Bff{v} = 0$, is
\[
	\Bff{a}_2 = -\ii(1-\Bff{K}_1\Bff{K}_2)^{-1}(-1 + 2\Bff{K}_1 - \Bff{K}_1 \Bff{K}_2)\Im \Bff{v}.
\]
Therefore, again with $\Re \Bff{v} = 0$,
\[
	\begin{aligned}
	\Bff{a}_2 - \Bff{K}_1\Bff{a}_1 &= \ii (1-\Bff{K}_1 \Bff{K}_2)^{-1}(1-2\Bff{K}_1 + \Bff{K}_1\Bff{K}_2 + \Bff{K}_1 - 2\Bff{K}_1\Bff{K}_2 + \Bff{K}_1\Bff{K}_2\Bff{K}_1)\Im \Bff{v}
	\\ &= \ii(1 - \Bff{K}_1 \Bff{K}_2)^{-1}(1-\Bff{K}_1-\Bff{K}_1 \Bff{K}_2 + \Bff{K}_1 \Bff{K}_2\Bff{K}_1)\Im \Bff{v}
	\\ &= \ii(1-\Bff{K}_1)\Im \Bff{v} = (1-\Bff{K}_1)\Bff{v}.
	\end{aligned}
\]
This proves \eqref{eq_a2a1v}.

Having established \eqref{eq_SVD_1}, \eqref{eq_SVD_2}, and \eqref{eq_SVD_3} on the level of canonical transformations, all that remains is to identify the constant coming from Proposition \ref{prop_crossing}. As in that proposition, let
\[
	\begin{aligned}
	\Bff{v}_1 &= (1-\Bff{K}_1^{-1})^{-1}\Bff{a}_1,
	\\ \Bff{v}_2 &= (1-\Bff{K}_1)^{-1}\Bff{a}_2.
	\end{aligned}
\]
From \eqref{eq_a1a2_SVD_3}, $\Bff{v} = \Bff{v}_1 + \Bff{v}_2$. Then, using Proposition \ref{prop_crossing} and \eqref{eq_shift_composition},
\begin{equation}\label{eq_proof_thm_deg2}
	\begin{aligned}
	\mathcal{S}_{\Bff{a}_2} \ee^{-\ii Q} \mathcal{S}_{\Bff{a}_1}^{-1} &= \ee^{-\frac{\ii}{2}\sigma(\Bff{a}_1, \Bff{v}_1)}\mathcal{S}_{\Bff{a}_2}\mathcal{S}_{\Bff{v}_1}\ee^{-\ii Q}\mathcal{S}_{\Bff{v}_1}^{-1}
	\\ &= \ee^{-\frac{\ii}{2}\sigma(\Bff{a}_1, \Bff{v}_1) + \ii \sigma(\Bff{a}_2, \Bff{v}_1)}\mathcal{S}_{\Bff{v}_1}\mathcal{S}_{\Bff{a}_2}\ee^{-\ii Q} \mathcal{S}_{\Bff{v}_1}^{-1}
	\\ &= \ee^{-\frac{\ii}{2}\sigma(\Bff{a}_1, \Bff{v}_1) + \ii \sigma(\Bff{a}_2, \Bff{v}_1) + \frac{\ii}{2}\sigma(\Bff{a}_2, \Bff{v}_2)}\mathcal{S}_{\Bff{v}_1}\mathcal{S}_{\Bff{v}_2} \ee^{-\ii Q} \mathcal{S}_{\Bff{v}_2}^{-1}\mathcal{S}_{\Bff{v}_1}^{-1}
	\\ &= \ee^{-\frac{\ii}{2}\sigma(\Bff{a}_1, \Bff{v}_1) + \ii \sigma(\Bff{a}_2, \Bff{v}_1) + \frac{\ii}{2}\sigma(\Bff{a}_2, \Bff{v}_2)}\mathcal{S}_{\Bff{v}} \ee^{-\ii Q} \mathcal{S}_{\Bff{v}}^{-1}
	\\ &= \ee^{-\frac{\ii}{2}\sigma(\Bff{a}_1, \Bff{v}_1) + \ii \sigma(\Bff{a}_2, \Bff{v}_1) + \frac{\ii}{2}\sigma(\Bff{a}_2, \Bff{v}_2)}\ee^{-\ii P}.
	\end{aligned}
\end{equation}
If we reverse the order, pushing $\Bff{a}_2$ across $\ee^{-\ii Q}$ before pushing $\Bff{a}_1$, we obtain the same exponent except that $\ii \sigma(\Bff{a}_2, \Bff{v}_1)$ is replaced by $\ii \sigma(\Bff{v}_2, \Bff{a}_1)$. We conclude that the two quantities are equal (which may be verified by a direct computation), and therefore
\[
	\begin{aligned}
	-\frac{\ii}{2}\sigma(\Bff{a}_1, \Bff{v}_1) + \ii \sigma(\Bff{a}_2, \Bff{v}_1) &+ \frac{\ii}{2}\sigma(\Bff{a}_2, \Bff{v}_2)
	\\ &= -\frac{\ii}{2}\sigma(\Bff{a}_1, \Bff{v}_1) + \frac{\ii}{2}\sigma(\Bff{v}_2, \Bff{a}_1) + \frac{\ii}{2} \sigma(\Bff{a}_2, \Bff{v}_1) + \frac{\ii}{2}\sigma(\Bff{a}_2, \Bff{v}_2)
	\\ &= \frac{\ii}{2}\left(-\sigma(\Bff{a}_1, \Bff{v}) + \sigma(\Bff{a}_2, \Bff{v})\right)
	\\ &= \frac{\ii}{2}\sigma(\Bff{a}_2 - \Bff{a}_1, \Bff{v}).
	\end{aligned}
\]
Solving for $\ee^{-\ii P}$ in \eqref{eq_proof_thm_deg2}, we conclude that
\begin{equation}\label{eq_shift_conclusion}
	\ee^{-\ii P} = \ee^{\frac{\ii}{2}\sigma(\Bff{v}, \Bff{a}_2 - \Bff{a}_1)}\ee^{-\ii Q}.
\end{equation}
This proves second part of Theorem \ref{thm_SVD}. Theorem \ref{thm_norm_deg2} follows from noting that
\[
	\Bff{a}_2 - \Bff{a}_1 = A\Im \Bff{v}
\]
when, writing again $\Bff{K} = \exp H_q$,
\begin{equation}\label{eq_def_A_matrix}
	A = \Im (\overline{\Bff{K}}^{-1})^{-1}(1- \Re \overline{\Bff{K}}^{-1}) + (\Im \Bff{K})^{-1}(1- \Re \Bff{K}).
\end{equation}
\end{proof}

The decomposition \eqref{eq_shift_conclusion}, combined with the similar argument at the end of Section \ref{ssec_norm_quadratic}, gives a decomposition of singular-value type.

\begin{theorem}\label{thm_SVD}
Let $q:\Bbb{R}^{2n}\to\Bbb{C}$ be a quadratic form such that $\Bff{K}_1 = \exp H_q$ is strictly positive, and fix $\Bff{v} \in \Bbb{C}^{2n}$. Let $Q = q^w$ and, for $p(x,\xi) = q((x,\xi) - \Bff{v})$, let $P = p^w$, and let $\ee^{-\ii Q}$ and $\ee^{-\ii P}$ be as in Section \ref{ssec_Fock}.

Recall also $\Bff{K}_2 = \overline{\Bff{K}_1}^{-1}$, the eigenvalues $\{\mu_j\}_{j=1}^n = \opnm{Spec} \Bff{K}_2 \Bff{K}_1 \cap (0, 1)$, and the vectors $\Bff{a}_1, \Bff{a}_2 \in \Bbb{R}^{2n}$ in \eqref{eq_a1_ImK} and \eqref{eq_a2_ImK}. Finally, define $Q_2$ such that $-\ii Q_2$ is the positive definite harmonic oscillator
\[
	-\ii Q_2 = -\frac{1}{2}\sum_{j=1}^n (\log \mu_j)(D_{x_j}^2 + x_j^2).
\]

Then there exist $\mathcal{U}, \mathcal{V}$ in the metaplectic group such that
\begin{equation}\label{eq_SVD_quad}
	\ee^{-\ii Q} = \mathcal{V} \ee^{-\frac{\ii}{2}Q_2}\mathcal{U}^*,
\end{equation}
and with the shifts defined in \eqref{eq_def_shift},
\[
	\ee^{-\ii P} = \ee^{\frac{\ii}{2}\sigma(\Bff{v}, \Bff{a}_2 - \Bff{a}_1)}\mathcal{S}_{\Bff{a}_2}\ee^{-\ii Q}\mathcal{S}_{\Bff{a}_1}^*.
\]
\end{theorem}

\begin{remark}\label{rem_selfadjoint_simplification}
We draw the reader's attention to the case $\overline{\Bff{K}_1}^{-1} = \Bff{K}_1$, which occurs precisely when $\ee^{-\ii Q}$ is self-adjoint. From \eqref{eq_a1_no_ImK}, and recalling that $\tanh (H_q/2) = (\Bff{K}+1)^{-1}(\Bff{K}-1)$,
\[
	\begin{aligned}
	\Bff{a}_1 &= \Re \Bff{v} - \ii (1-\Bff{K}^2)^{-1} (1-\Bff{K})^2 \Im \Bff{v}
	\\ &= \Re \Bff{v} - \ii (1 + \Bff{K})^{-1}(1-\Bff{K})\Im \Bff{v}
	\\ &= \Re \Bff{v} - \frac{1}{\ii}\tanh(H_q/2)\Im \Bff{v}.
	\end{aligned}
\]
Since $\Bff{a}_2$ is obtained by replacing $\Bff{K}$ by $\overline{\Bff{K}}^{-1}$ (which has no effect since we assumed these two are equal) and $\Bff{v}$ by $\overline{\Bff{v}}$, we see that
\[
	\Bff{a}_2 = \Re \Bff{v} + \frac{1}{\ii}\tanh(H_q/2)\Im \Bff{v}.
\]
Therefore, from \eqref{eq_shift_conclusion},
\[
	\ee^{-\ii P} = \ee^{-\sigma(\Bff{v}, \frac{1}{\ii}\tanh(H_q/2)\Im \Bff{v})}S_{\Bff{a}_2}\ee^{-\ii Q}S_{\Bff{a}_1}^{-1},
\]
and
\[
	\|\ee^{-\ii P}\| = \ee^{-\sigma(\Im \Bff{v}, \frac{1}{\ii}\tanh(H_q/2)\Im \Bff{v})}\|\ee^{-\ii Q}\|.
\]
Note that the exponent is the negative of the Mehler exponent \eqref{eq_Mehler}, taken at $\Im \Bff{v}$, which is real when $\overline{\Bff{K}}^{-1} = \Bff{K}$.
\end{remark}

\section{Structure of evolution operators generated by degree-2 Hamiltonians}\label{sec_generators_canonical}

This section is devoted the fundamental structure of the set of evolution operators used in the present work. We begin by recalling recalling the generalized solutions introduced in \cite{Aleman_Viola_2014b} when the generator is a linear perturbation of a supersymmetric quadratic form. We then show the equivalence of positivity or ellipticity conditions used in this work. Next, we dispense with the supersymmetry hypothesis by showing that it is implied by strict positivity of the Hamilton flow, and that conversely any strictly positive canonical transformation corresponds to the flow of a supersymmetric quadratic form. Using the FBI--Bargmann point of view, we then establish the extension of well-known Mehler formulas and Egorov relations for quadratic generators, in addition to the classical-quantum correspondence. Finally, we prove Theorem \ref{thm_gaussian_kernels}.

\subsection{Evolution operators via Fock spaces}\label{ssec_Fock}

We begin by recalling the maximal definition \cite{Aleman_Viola_2014b} of $\ee^{-\ii Q}$, when $Q = q^w$ for $q$ a supersymmetric quadratic form in the sense defined below. This is performed via an FBI--Bargmann reduction essentially due to \cite{Sjostrand_1974}. For further details on FBI--Bargmann transforms with quadratic phase, we refer the reader to \cite[Ch.~13]{Zworski_2012} or \cite[Ch.~12]{Sjostrand_LoR}. We see in Proposition \ref{prop_positivity_implies_supersymmetry} that a supersymmetry hypothesis would be superfluous in the context of this work, since it is implied by strict positivity of the Hamilton flow.

\begin{definition}\label{def_supersymmetric}
A quadratic form $q:\Bbb{R}^{2n}\to \Bbb{C}$ is supersymmetric if it may be written as
\[
	q(x,\xi) = B(\xi - G_+ x)\cdot (\xi - G_- x)
\]
for $B \in \Bbb{C}^{n\times n}$ any matrix and $G_+, G_- \in \Bbb{C}^{n\times n}$ symmetric matrices for which $\pm \Im G_\pm > 0$ in the sense of positive definite matrices.
\end{definition}

By \cite[Prop.~3.3]{Aleman_Viola_2014b}, a quadratic form $q$ is supersymmetric if and only if $Q = q^w$ can be reduced to $Mx\cdot\partial_x + \frac{1}{2}\tr M$ via an appropriate FBI--Bargmann transform $\mathfrak{T}$. In more detail, this latter condition means that there exists both a complex linear canonical transformation $\Bff{T}$ such that 
\begin{equation}\label{eq_q_tilde}
	\tilde{q}(x,\xi) := q\circ \Bff{T}^{-1}(x,\xi) = Mx\cdot \ii \xi
\end{equation}
for some matrix $M \in \Bbb{M}_{n\times n}(\Bbb{C})$, and a unitary map
\[
	\begin{aligned}
	\mathfrak{T}u(x) &= c_\varphi \int \ee^{\ii \varphi(x,y)}\,u(y)\,\dd y,
	\\ \mathfrak{T}&:L^2(\Bbb{R}^n) \to H_\Phi(\Bbb{C}^n) = \opnm{Hol}(\Bbb{C}^n) \cap L^2(\Bbb{C}^n, \ee^{-2\Phi(x)}\,\dd \Re x\,\dd \Im x),
	\end{aligned}
\]
where $\Phi: \Bbb{C}^n \to [0, \infty)$ is real-quadratic and strictly convex,
for which, for any $a \in \mathscr{S}'(\Bbb{R}^{2n})$,
\begin{equation}\label{eq_Egorov_FBI}
	\mathfrak{T}a^w\mathfrak{T}^* = (a \circ \Bff{T}^{-1})^w.
\end{equation}

As a result, with $Q = q^w$ and $\tilde{Q} = \tilde{q}^w$,
\begin{equation}\label{eq_Egorov_q_tilde}
	\mathfrak{T}Q\mathfrak{T}^* = \tilde{Q} = Mx\cdot \partial_x + \frac{1}{2}\tr M.
\end{equation}
Note also that composition with $\Bff{T}$ induces a similarity relation for Hamilton vector fields,
\begin{equation}\label{eq_similarity_FBI}
	H_{\tilde{q}} = \Bff{T}H_q\Bff{T}^{-1},
\end{equation}
and $H_{\tilde{q}}$ takes the simple form
\[
	H_{\tilde{q}} = \left(\begin{array}{cc} \ii M & 0 \\ 0 & -\ii M^\top\end{array}\right).
\]
The same similarity relation also simplifies the Hamilton flow: if $\Bff{K} = \exp H_q$ and $\tilde{\Bff{K}} = \exp H_{\tilde{q}}$, then
\begin{equation}\label{eq_K_FBI}
	\tilde{\Bff{K}} = \Bff{T}\Bff{K}\Bff{T}^{-1} = \left(\begin{array}{cc}\ee^{\ii M} & 0 \\ 0 & \ee^{-\ii M^\top}\end{array}\right).
\end{equation}
This classical fact, that conjugation by $\mathfrak{T}$ serves to block-diagonalize $H_q$, is the cornerstone of the analysis here and in \cite{Aleman_Viola_2014b}. For elliptic complex-valued quadratic forms, this technique comes from \cite{Sjostrand_1974}.

\begin{example}\label{ex_Bargmann}
The classical Bargmann transform \cite[Eq.~(2.1)]{Bargmann_1961}
\[
	\mathfrak{B}_0 = \pi^{-3n/4}\int \ee^{-\frac{1}{2}(x^2 + y^2)+\sqrt{2}xy}u(y)\,\dd y
\]
is a unitary map onto the space of holomorphic functions $u$ for which $u(x)\ee^{-|x|^2/2} \in L^2(\Bbb{C}^n_x, \dd \Re x \, \dd \Im x)$. That is, in this case, $\Phi(x) = \frac{1}{2}|x|^2$.

Furthermore, $\mathfrak{B}_0$ quantizes the complex linear canonical transformation
\begin{equation}\label{eq_def_B0}
	\Bff{B}_0(x,\xi) = \frac{1}{\sqrt{2}}(x-\ii \xi, -\ii x + \xi).
\end{equation}
Since, when $q_0(x,\xi) = \frac{1}{2}(\xi^2 + x^2)$, one has $q_0\circ \Bff{B}_0^{-1}(x,\xi) = x\cdot \ii \xi$, the Bargmann transform reduces the harmonic oscillator $Q_0 = q_0^w$ to
\[
	\mathfrak{B}_0 Q_0 \mathfrak{B}_0^* = x\cdot \partial_x + \frac{n}{2}.
\]
\end{example}

On $H_\Phi$, one can show \cite[Thm.~2.9,~2.12]{Aleman_Viola_2014b} that $\tilde{Q}$ has a complete set of generalized eigenfunctions which are homogeneous polynomials, and the span of this set is a core for the evolution operator
\begin{equation}\label{eq_FBI-side_evolution}
	\exp(-\ii \tilde{Q}) = \exp\left(-\ii Mx\cdot \partial_x - \frac{\ii}{2}\tr M\right)u(x) = \ee^{-\frac{\ii}{2}\tr M}u(\ee^{-\ii M}x)
\end{equation}
with maximal domain
\[
	\mathcal{D}(\exp(-\ii \tilde{Q})) = \{u \in H_\Phi \::\: u(\ee^{-\ii M}\cdot)\in H_\Phi\}.
\]
We then define
\begin{equation}\label{eq_FBI_evolution}
	\exp(-\ii Q) = \mathfrak{T}^*\ee^{-\ii \tilde{Q}}\mathfrak{T}
\end{equation}
with the closed dense maximal domain
\[
	\mathcal{D}(\ee^{-\ii Q}) = \mathfrak{T}^*(\mathcal{D}(\ee^{-\ii \tilde{Q}})).
\]

We recall \cite[Thm.~2.9]{Aleman_Viola_2014b} that this operator is compact on $L^2(\Bbb{R}^n)$ if and only if
\begin{equation}\label{eq_FBI_compactness}
	\Phi(\ee^{\ii M}x) - \Phi(x) > 0, \quad \forall x \in \Bbb{C}^n\backslash\{0\}.
\end{equation}
The operator is bounded if and only if the non-strict version of this inequality holds.

We also recall, following \cite[Prop.~2.23]{Aleman_Viola_2014b}, how to compute the Schr\"odinger evolution associated with $P = \Op^w(q(\Bff{z} - \Bff{v}))$ on the FBI--Bargmann side. Continuing to let $\mathfrak{T}$ be a FBI--Bargmann transform adapted to $q$, let
\[
	\mathfrak{w} = \Bff{T} \Bff{v} = (\mathfrak{w}_x, \mathfrak{w}_\xi).
\]
Therefore $\mathfrak{T} \mathcal{S}_{\Bff{v}}\mathfrak{T}^* = \mathcal{S}_{\mathfrak{w}}$ and, where $\Bff{z} = (x,\xi)$,
\[
	\tilde{p}(\Bff{z}) := (p\circ \Bff{T}^{-1}) = \tilde{q}(\Bff{z} - \mathfrak{w})) = M(x - \mathfrak{w}_x) \cdot \ii (\xi - \mathfrak{w}_\xi).
\]
Letting $\tilde{P} = \tilde{p}^w$, we define
\[
	\ee^{-\ii P} = \mathfrak{T}^* \ee^{-\ii \tilde{P}} \mathfrak{T}
\]
where
\[
	\begin{aligned}
	\ee^{-\ii \tilde{P}}u(x) &= \mathcal{S}_{\mathfrak{w}} \ee^{-\ii \tilde{Q}}\mathcal{S}_{\mathfrak{w}}^{-1}u(x)
	\\ &= \mathcal{S}_{\mathfrak{w}} \ee^{-\ii \tilde{Q}}\ee^{-\ii \mathfrak{w}_\xi \cdot x - \frac{\ii}{2} \mathfrak{w}_\xi \cdot \mathfrak{w}_x} u(x+\mathfrak{w}_x)
	\\ &= \mathcal{S}_{\mathfrak{w}} \ee^{-\frac{\ii}{2}\tr M - \ii \mathfrak{w}_\xi \cdot \ee^{-\ii M}x - \frac{\ii}{2}\mathfrak{w}_\xi \cdot \mathfrak{w}_x} u(\ee^{-\ii M}x+ \mathfrak{w}_x)
	\\ &= \ee^{-\frac{\ii}{2} \tr M + \ii \mathfrak{w}_\xi \cdot x - \ii \mathfrak{w}_\xi \ee^{-\ii M}(x-\mathfrak{w}_x) - \ii \mathfrak{w}_\xi \cdot \mathfrak{w}_x} u(\ee^{-\ii M}(x-\mathfrak{w}_x) + \mathfrak{w}_x)
	\\ &= \ee^{-\frac{\ii}{2} \tr M + \ii \mathfrak{w}_\xi \cdot (1-\ee^{-\ii M})(x - \mathfrak{w}_x)} u(\ee^{-\ii M}x + (1-\ee^{-\ii M})\mathfrak{w}_x).
	\end{aligned}
\]
Note that an advantage of working on the FBI--Bargmann side is that $\mathcal{S}_{\mathfrak{w}}$ is defined on any function in $H_\Phi$, though the image may not belong to $H_\Phi$.

With this definition, note that
\begin{multline*}
	\|\ee^{-\ii \tilde{P}}u(x)\|_{H_\Phi}^2 = \int |\ee^{-\ii \tilde{P}}u(x)|^2 \ee^{-2\Phi(x)}\,\dd L(x)
	\\ = c\int |u(y)|^2 \exp\left( 2 \Im(\mathfrak{w}_\xi \cdot (1 - \ee^{\ii M})y) - 2\Phi(\ee^{\ii M}(y - (1-\ee^{-\ii M}))\mathfrak{w}_x)\right)\,\dd L(y),
\end{multline*}
for
\[
	c = \ee^{\Im\tr M + 2\Im(\mathfrak{w}_\xi \cdot (\ee^{\ii M} - 1)\mathfrak{w}_x)}.
\]
We therefore define the weight
\[
	\begin{aligned}
	\tilde{\Phi}(x) &= \Phi(\ee^{\ii M}(x - (1-\ee^{-\ii M}))\mathfrak{w}_x) -\Im(\mathfrak{w}_\xi \cdot (1 - \ee^{\ii M})x) 
	\\ &= \Phi(\ee^{\ii M}x) + \BigO(|x|), \quad |x|\to \infty.
	\end{aligned}
\]
The assumption that $\ee^{-\ii Q}$ is compact, via \eqref{eq_FBI_compactness}, is therefore a sufficient condition to ensure that $\ee^{-\ii P}$ is compact as well; see \cite[Prop.~2.23]{Aleman_Viola_2014b}.

\subsection{Positivity and boundedness}

In this work, we use three possible ellipticity criteria for an evolution operator: that the canonical transformation is strictly positive; that the Mehler formula \eqref{eq_Mehler} is integrable on $\Bbb{R}^{2n}$; and that the evolution operator is compact, which can be determined via \eqref{eq_FBI_compactness}. The goal of this section is to show that these three are identical, and also to show that this is equivalent to integrability of any Gaussian kernel associated with a canonical transformation.

Writing $\Bff{K} = \exp H_q$, the exponent in the Mehler formula \eqref{eq_Mehler} is
\[
	\sigma(\Bff{z}, \frac{1}{\ii}\tanh(H_q)\Bff{z}) = \sigma(\Bff{z}, \ii(1+\Bff{K})^{-1}(1-\Bff{K})\Bff{z}), \quad \Bff{z} \in \Bbb{R}^{2n}.
\]
The Mehler formula is integrable if and only if the real part of this exponent is negative definite on $\Bbb{R}^{2n}$. We prove the more general fact that the associated Hermitian form is negative definite on $\Bbb{C}^{2n}$.

\begin{proposition}\label{prop_Mehler_integrability}
Let $\Bff{K}$ be a linear canonical transformation $\Bff{K}$. Then the following three conditions are equivalent:
\begin{itemize} 
	\item $-1 \notin \opnm{Spec}\Bff{K}$ and
	\begin{equation}\label{eq_Mehler_integrability}
		\Re \sigma(\overline{\Bff{z}}, \ii(1+\Bff{K})^{-1}(1-\Bff{K}) \Bff{z}) < 0, \quad \forall \Bff{z} \in \Bbb{C}^{2n}\backslash \{0\},
	\end{equation}
	\item $-1 \notin \opnm{Spec}\Bff{K}$ and
	\[
		\Re \sigma(\Bff{z}, \ii(1+\Bff{K})^{-1}(1-\Bff{K}) \Bff{z}) < 0, \quad \forall \Bff{z} \in \Bbb{R}^{2n}\backslash \{0\},
	\]
	and
	\item $\Bff{K}$ is strictly positive.
\end{itemize}
\end{proposition}

\begin{proof}
The second and third conditions are equivalent because $(1+\Bff{K})^{-1}(1-\Bff{K})$ is antisymmetric with respect to $\sigma$, and therefore when $\Bff{z} = \Bff{x} + \ii \Bff{y}$ for $\Bff{x}, \Bff{y} \in \Bbb{R}^{2n}$,
\[
	\sigma(\overline{\Bff{z}}, \ii(1+\Bff{K})^{-1}(1-\Bff{K}) \Bff{z}) = \sigma(\Bff{x}, \ii(1+\Bff{K})^{-1}(1-\Bff{K}) \Bff{x}) + \sigma(\Bff{y}, \ii(1+\Bff{K})^{-1}(1-\Bff{K}) \Bff{y}).
\]
We work with the second condition because it allows us to make complex linear changes of variables.

Strict positivity implies that $\Bff{K}$ has no eigenvalue of modulus one; \emph{a fortiori}, if $\Bff{K}$ is strictly positive, $-1 \notin \opnm{Spec}\Bff{K}$.

Let $\Bff{w} = (1+\Bff{K})^{-1}$, so
\[
	\begin{aligned}
	\sigma(\overline{\Bff{z}}, \ii(1+\Bff{K})^{-1}(1-\Bff{K}) \Bff{z}) &= \ii \sigma(\overline{(1+\Bff{K})\Bff{w}}, (1-\Bff{K})\Bff{w})
	\\ &= \ii\left(\sigma(\overline{\Bff{w}}, \Bff{w}) - \sigma(\overline{\Bff{Kw}}, \Bff{Kw}) + \sigma(\overline{\Bff{w}}, \Bff{Kw}) - \sigma(\overline{\Bff{Kw}}, \Bff{w})\right)
	\\ &= \ii\left(\sigma(\overline{\Bff{w}}, \Bff{w}) - \sigma(\overline{\Bff{Kw}}, \Bff{Kw}) + 2\Re(\sigma(\overline{\Bff{w}}, \Bff{Kw}))\right).
	\end{aligned}
\]
Seeing now that
\[
	\Re \sigma(\overline{\Bff{z}}, \ii(1+\Bff{K})^{-1}(1-\Bff{K}) \Bff{z}) = \ii\left(\sigma(\overline{\Bff{w}}, \Bff{w}) - \sigma(\overline{\Bff{Kw}}, \Bff{Kw})\right),
\]
the proposition is obvious from Definiton \ref{def_positive_canonical} of strict positivity of $\Bff{K}$.
\end{proof}

Our third natural ellipticity condition is that $\ee^{-\ii Q}$ should be compact, which is determined by \eqref{eq_FBI_compactness}. This condition is also equivalent to strict positivity of the associated canonical transformation.

\begin{proposition}\label{prop_compact_strictly_pos}
For $q$ a supersymmetric quadratic form as defined in Definition \ref{def_supersymmetric} and setting $Q = q^w$, the operator $\exp(-\ii Q)$ from \eqref{eq_FBI_evolution} is compact if and only if $\exp H_q$ is strictly positive as in Definition \ref{def_positive_canonical}.
\end{proposition}

\begin{proof}
Recall the FBI--Bargmann side flow $\tilde{\Bff{K}}$ in \eqref{eq_K_FBI}. If $-1 \in \opnm{Spec}\tilde{\Bff{K}}$ then $-1 \in \opnm{Spec}(\ee^{\ii M})$, so there exists some $x_0 \in \Bbb{C}^n \backslash \{0\}$ such that $\ee^{\ii M}x_0 = -x_0$. Because $\Phi$ is quadratic, if this occurs, then \eqref{eq_FBI_compactness} clearly cannot hold.

Assuming therefore that $-1 \notin \opnm{Spec}\Bff{K}$, by Proposition \ref{prop_Mehler_integrability}, $\Bff{K} = \exp H_q$ is strictly positive if and only if
\[
	\sigma((x,\xi), \ii(1+\tilde{\Bff{K}})^{-1}(1 - \tilde{\Bff{K}})(x,\xi)) = 2\ii(1+\ee^{\ii M})^{-1}(1-\ee^{\ii M})x\cdot \xi
\]
decays along 
\[
	\Lambda_\Phi = \Bff{T}(\Bbb{R}^{2n}) = \left\{\left(x, \frac{2}{\ii}\Phi'_x(x)\right)\right\}_{x \in \Bbb{C}^n}.
\]
(See, for instance, \cite[Thm.~13.5]{Zworski_2012} for a discussion of $\Lambda_\Phi$.) This is true if and only if
\begin{equation}\label{eq_compact_FBI_Mehler}
	\Re\left(2\ii(\ee^{\ii M}+1)^{-1}(\ee^{\ii M}-1)x\cdot \frac{2}{\ii}\Phi'_x(x)\right) > 0, \quad x \in \Bbb{C}^n \backslash \{0\}.
\end{equation}

Note that, since $\Phi$ is a real-quadratic form on $\Bbb{C}^n$, we can write
\[
	\Phi(x,y) = \Re \left(x \cdot \Phi'(y)\right)
\]
as the unique real-valued symmetric quadratic form on $\Bbb{C}^{2n}$ such that $\Phi(x,x) = \Phi(x)$. We then make the change of variables $y = (\ee^{\ii M} + 1)^{-1}x$ to compute
\[
	\begin{aligned}
	4 \Re((\ee^{\ii M}+1)^{-1}&(\ee^{\ii M}-1)x\cdot \Phi'_x(x)) = 4 \Phi((\ee^{\ii M}+1)^{-1}(\ee^{\ii M}-1)x, x)
	\\ &= 4 \Phi((\ee^{\ii M}-1)y, (\ee^{\ii M}+1)y)
	\\ &= 4 \Phi(\ee^{\ii M}y) + 4\Phi(\ee^{\ii M}y, y) - 4\Phi(y, \ee^{\ii M}y) - 4\Phi(y)
	\\ &= 4\left(\Phi(\ee^{\ii M}y) - \Phi(y)\right).
	\end{aligned}
\]
This computation makes it clear that \eqref{eq_compact_FBI_Mehler} holds if and only if \eqref{eq_FBI_compactness} holds, proving the proposition.
\end{proof}

Let $\varphi(x,y):\Bbb{R}^{2n}\to\Bbb{C}$ be a holomorphic quadratic form for which $\det \varphi''_{xy} \neq 0$. The integral operator
\[
	\mathfrak{T}_\varphi u(x) = \int_{\Bbb{R}^n}\ee^{\ii \varphi(x,y)}u(y)\,\dd y
\]
is associated with the canonical transformation
\begin{equation}\label{eq_K_phi_kernel}
	(y, -\varphi'_y(x,y)) \stackrel{\Bff{K}_\varphi}{\mapsto} (x, \varphi'_x(x,y)),
\end{equation}
or equivalently
\begin{equation}\label{eq_def_K_phi}
	\Bff{K}_\varphi = \left(\begin{array}{cc} -(\varphi''_{yx})^{-1} \varphi''_{yy} & -(\varphi''_{yx})^{-1} \\ \varphi''_{xy} - \varphi''_{xx} (\varphi''_{yx})^{-1} \varphi''_{yy} & -\varphi''_{xx} (\varphi''_{yx})^{-1}\end{array}\right).
\end{equation}
In the following proposition, we confirm via a standard computation that a Gaussian kernel for which $\det \varphi''_{xy} \neq 0$ is associated with a canonical transformation which is strictly positive if and only if the Gaussian kernel is non-degenerate (that is, integrable).

\begin{proposition}\label{prop_Gaussian_nondeg_iff_positive}
If $\varphi:\Bbb{R}^n_x \times \Bbb{R}^n_y \to \Bbb{C}$ is a quadratic form for which $\det \varphi''_{xy} \neq 0$, then $\Bff{K}_\varphi$ in \eqref{eq_def_K_phi} is strictly positive if and only if $\Im \varphi$ is positive definite.
\end{proposition}

\begin{proof}
The adjoint of $\mathfrak{T}_\varphi$ is $\mathfrak{T}_{\varphi^*}$ where
\[
	\varphi^*(x,y) = -\overline{\varphi(\overline{y}, \overline{x})},
\]
so it comes as no surprise that, as may be verified directly,
\[
	\Bff{K}_{\varphi^*} = \overline{\Bff{K}_\varphi}^{-1}.
\]
We then write
\[
	\Bff{A} := \Bff{K}_{\varphi^*}\Bff{K}_\varphi = \left(\begin{array}{cc} A_{11} & A_{12} \\ A_{21} & A_{22}\end{array}\right),
\]
where
\[
	\begin{aligned}
	A_{11} &= (\overline{\varphi''_{xy}})^{-1}\left( \varphi''_{xy} - 2\ii \Im(\varphi''_{xx})(\varphi''_{yx})^{-1}\varphi''_{yy}\right)
	\\ A_{12} &= (\overline{\varphi''_{xy}})^{-1} (-2\ii \Im \varphi''_{xx})(\varphi''_{yx})^{-1}
	\\ A_{21} &= 2\ii \left(\Im \left(\overline{\varphi''_{yx}}(\varphi''_{yx})^{-1}\varphi''_{yy} \right) + \overline{\varphi''_{yy}}(\overline{\varphi''_{xy}})^{-1} \Im (\varphi''_{xx})(\varphi''_{yx})^{-1}\varphi''_{yy}\right)
	\\ A_{22} &= \left(\overline{\varphi''_{yx}} + 2\ii \overline{\varphi''_{yy}}(\overline{\varphi''_{xy}})^{-1} \Im (\varphi''_{xx})\right)(\varphi''_{yx})^{-1}.
	\end{aligned}
\]

The map $\Bff{K}_\varphi$ is strictly positive if and only if
\[
	\begin{aligned}
	\Psi(x,\xi) &:= \ii \sigma((\overline{x}, \overline{\xi}), (\overline{\Bff{K}_\varphi}^{-1}\Bff{K}_\varphi - 1)(x,\xi)) 
	\\ &= (\overline{x}, \overline{\xi})\cdot \left(\begin{array}{cc} -\ii A_{21} & -\ii (A_{22} - 1) \\ \ii (A_{11} - 1) & \ii A_{12}\end{array}\right)\left(\begin{array}{c} x \\ \xi \end{array}\right)
	\end{aligned}
\]
is a positive definite quadratic form. (Note that the matrix is Hermitian.) Since
\[
	\Psi(0, \varphi''_{yx}\eta) = 2\overline{\eta} \cdot (\Im \varphi''_{xx})\eta,
\]
strict positivity of $\Bff{K}_\varphi$ implies that $\Im \varphi''_{xx}$ is positive definite.

Completing the square,
\[
	\begin{aligned}
	\Psi(x,\xi) &= (\overline{\xi} + \overline{A_{12}^{-1}(A_{11}-1)x})\cdot \ii A_{12}(\xi + A_{12}^{-1}(A_{11}-1)x) 
	\\ & \quad + \overline{x}\cdot \ii((A_{22}-1)A_{12}^{-1}(A_{11}-1)-A_{21})x.
	\end{aligned}
\]
This shows that $\Psi$ is positive definite if and only if this second term is positive definite, and a direct computation shows that
\[
	\ii((A_{22}-1)A_{12}^{-1}(A_{11}-1)-A_{21}) = 2\Im(\varphi''_{yy}) - 2\Im(\varphi''_{yx})(\Im \varphi''_{xx})^{-1} \Im(\varphi''_{xy}).
\]
By the same argument via difference of squares, this matrix and $\Im \varphi''_{xx}$ are both positive definite if and only if $\Im \varphi$ is a positive definite quadratic form; this proves the proposition.
\end{proof}

\subsection{Supersymmetry and strictly positive linear canonical transformations}\label{ssec_supersymmetric_onto}

In this work we focus on quadratic forms yielding strictly positive Hamilton flows. In this section, we show that this hypothesis implies supersymmetric structure, and therefore allows us to apply the definition of $\ee^{-\ii P}$ in Section \ref{ssec_Fock}. Furthermore, we see that every strictly positive linear canonical transformation is the Hamilton flow of some quadratic form, which is necessarily supersymmetric.

\begin{proposition}\label{prop_positivity_implies_supersymmetry}
Let $q:\Bbb{R}^{2n} \to \Bbb{C}$ be a quadratic form such that $\exp H_q$ is strictly positive. Then $q$ is supersymmetric in the sense of Definition \ref{def_supersymmetric}.
\end{proposition}

\begin{proof} By \cite[Prop.~3.3]{Aleman_Viola_2014b}, it is enough to show that $H_q$ has invariant subspaces $\Lambda^+, \Lambda^-$ which are positive and negative Lagrangian planes. The only fact \cite[Prop.~3.3]{Sjostrand_1974} we need about this type of invariant subspace is that, if $q_1$ is a quadratic form on $\Bbb{R}^{2n}$ such that $\Re q_1$ is positive definite, then we may define the two subspaces $\Lambda^\pm(q_1)$ as the sum of generalized eigenspaces
\begin{equation}\label{eq_Lambda_pm_elliptic}
	\Lambda^\pm(q_1) = \bigoplus_{\pm \Im \lambda > 0} \ker(\frac{1}{2}H_{q_1} - \lambda)^n.
\end{equation}
(Since $n = \dim \Lambda^\pm(q_1)$, we know that $\ker(\frac{1}{2}H_{q_1} - \lambda)^n$ always gives the generalized eigenspace of $\frac{1}{2}H_{q_1}$ corresponding to an eigenvalue $\lambda$.)

Let 
\begin{equation}\label{eq_q1_Mehler_canonical}
	q_1(\Bff{z}) = \sigma(\Bff{z}, -\ii(1+\Bff{K})^{-1}(1-\Bff{K})\Bff{z}).
\end{equation}
By Proposition \ref{prop_Mehler_integrability}, $\Re q_1$ is positive definite on $\Bbb{R}^{2n}$, so we may apply \eqref{eq_Lambda_pm_elliptic}. Furthermore, 
\begin{equation}\label{eq_Hq1_ssym}
	\frac{1}{2}H_{q_1}= -\ii(1+\Bff{K})^{-1}(1-\Bff{K}),
\end{equation}
which implies
\[
	\Bff{K} = (1+\frac{\ii}{2} H_{q_1})^{-1}(1-\frac{\ii}{2} H_{q_1}),
\]
and the linear fractional transformation $f(\zeta) = (1+\ii \zeta)^{-1}(1-\ii \zeta)$ maps $\{\pm \Im \zeta > 0\}$ to $\{|\zeta|^{\pm 1} > 1\}$. (It suffices to check that $f(\pm 1) = \mp \ii$ and $f(0) = 1$, so the real axis is mapped to the unit circle, and that $f(-\ii) = 0$.) Since $\Bff{K} = \exp H_q$, we compute
\[
	\begin{aligned}
	\Lambda^\pm(q_1) &= \bigoplus_{\pm \Im \lambda > 0} \ker(\frac{1}{2}H_{q_1} - \lambda)^n
	\\ &= \bigoplus_{|\lambda|^{\pm 1} > 1} \ker(\Bff{K} - \lambda)^n
	\\ &= \bigoplus_{\pm \Re \lambda > 0} \ker(H_q - \lambda)^n.
	\end{aligned}
\]

Direct sums of generalized eigenspaces of $H_q$ are $H_q$-invariant, so the hypotheses of \cite[Prop.~3.3]{Aleman_Viola_2014b} are satisfied and Proposition \ref{prop_positivity_implies_supersymmetry} is proved.
\end{proof}

\begin{remark}
The fact that $\opnm{Spec} H_q|_{\Lambda^+} \subset \{\Re \lambda > 0\}$ is a necessary and sufficient condition for existence of some $s_0 \geq 0$ such that $\ee^{-\ii s Q}$ is compact for all $s > s_0$. This is more or less immediate from the compactness condition \eqref{eq_FBI_compactness} and the fact \cite[Cor.~3.4]{Aleman_Viola_2014b} that, for $M$ in \eqref{eq_q_tilde},
\[
	\opnm{Spec}M = -\ii \opnm{Spec}H_q|_{\Lambda^+} = \ii \opnm{Spec}H_q|_{\Lambda^-}.
\]
\end{remark}

Next, we prove that an arbitrary strictly positive canonical transformation corresponds to the Hamilton flow of a supersymmetric quadratic form (in fact, infinitely many). As a result, the set of Schr\"odinger evolutions of supersymmetric quadratic forms suffices to describe the set of Fourier integral operators associated to strictly positive canonical transformations in the sense of \cite{Hormander_1995}. We perform the proof on the FBI--Bargmann side, because there we can make a natural choice of $\log \Bff{K}$.

\begin{proposition}\label{prop_supersymmetric_onto}
Let $\Bff{K}:\Bbb{C}^{2n} \to \Bbb{C}^{2n}$ be a strictly positive linear canonical transformation. Then there exists a supersymmetric quadratic form $q$ such that $\exp H_q = \Bff{K}$.
\end{proposition}

\begin{proof}
Let $\Bff{K}$ be a strictly positive canonical transformation, and let $q_1$ be the quadratic form with positive definite real part from \eqref{eq_q1_Mehler_canonical}. Since $q_1$ is supersymmetric, as in Section \ref{ssec_Fock} we may choose some FBI--Bargmann transform $\mathfrak{T}$ associated with a canonical transformation $\Bff{T}$ such that
\[
	\tilde{q}_1(x,\xi) := (q_1 \circ \Bff{T}^{-1})(x,\xi) = M_1x\cdot \ii \xi.
\]
But on the other hand, because $\Bff{T}$ is canonical, when
\[
	\tilde{\Bff{K}} = \Bff{T} \Bff{K}\Bff{T}^{-1},
\]
we have
\[
	\tilde{q}_1(\Bff{z}) = \sigma(\Bff{z}, -\ii(1+\tilde{\Bff{K}})^{-1}(1-\tilde{\Bff{K}})\Bff{z}).
\]
Since $\tilde{\Bff{K}}$ is canonical, $(1+\tilde{\Bff{K}})^{-1}(1-\tilde{\Bff{K}})$ is antisymmetric with respect to $\sigma$, so 
\begin{equation}\label{eq_M1_tildeK}
	H_{\tilde{q}_1} = -2\ii (1+\tilde{\Bff{K}})^{-1}(1-\tilde{\Bff{K}}) = \left(\begin{array}{cc} \ii M_1 & 0 \\ 0 & -\ii M_1^\top\end{array}\right).
\end{equation}

By \cite[Cor.~3.4]{Aleman_Viola_2014b},
\[
	\Spec \frac{1}{2} M_1 \subset \Spec \frac{1}{2\ii}H_{q_1} = \Spec (\Bff{K}+1)^{-1}(\Bff{K}-1),
\]
which transparently does not contain $\pm 1$. Solving for $\tilde{\Bff{K}}$ in \eqref{eq_M1_tildeK},
\[
	\tilde{\Bff{K}} = \left(\begin{array}{cc} M_2 & 0 \\ 0 & (M_2^\top)^{-1}\end{array}\right)
\]
when
\[
	M_2 = (1-M_1/2)^{-1}(1+M_1/2).
\]
Since $M_2$ is invertible, we may define via the Jordan normal form
\begin{equation}\label{eq_log_FBI_side}
	M = \frac{1}{\ii}\log M_2.
\end{equation}

Setting
\[
	\tilde{q}(x,\xi) = Mx\cdot \ii \xi,
\]
we have
\[
	\exp H_{\tilde{q}} = \left(\begin{array}{cc} \exp \log M_2 & 0 \\ 0 & \exp(-\log M_2^\top)\end{array}\right) = \tilde{\Bff{K}}.
\]
Therefore let
\[
	q = \tilde{q}\circ \Bff{T}.
\]
By the induced similarity relation on Hamilton maps, $\exp H_q = \Bff{K}$, and $q$ is supersymmetric by Proposition \ref{prop_positivity_implies_supersymmetry}. Note that, by \cite[Prop.~3.3]{Aleman_Viola_2014b}, supersymmetry is already implied by the existence of a reduction to $Mx\cdot \ii \xi$ given by $\Bff{T}$. This proves the proposition.
\end{proof}

\begin{remark}
The matrix $M$ in \eqref{eq_log_FBI_side} may be modified on any Jordan block by adding $2\pi$ to the associated eigenvalue, so there are infinitely many quadratic forms corresponding to any given strictly positive linear canonical transformation. One can freely modify the sign associated with $\ee^{-\ii Q}$ unless the size of every Jordan block in the Jordan normal form $M_2$ is even. Note that it is possible to find a strictly positive canonical transformation $\Bff{K}$ such that that only one such quadratic form obeys the ellipticity condition $\Im q \leq 0$, as can be seen by taking $\exp H_{\eps q_\theta}$ from Example \ref{ex_Davies} with $\theta \neq 0$ and $\eps > 0$ sufficiently small depending on $\theta$. 
\end{remark}

Proposition \ref{prop_supersymmetric_onto} in the special case $\overline{\Bff{K}}^{-1} = \Bff{K}$ deserves particular attention due to its importance in Section \ref{ssec_norm_quadratic}. The hypotheses we specify are essentially to rule out a situation like $\Bff{K} = -1$ which is associated with the harmonic oscillator evolution $\ee^{-\ii \pi Q_0}u(x) = -\ii u(-x)$.

\begin{proposition}\label{prop_real_log}
Let $\Bff{K}$ be a strictly positive linear canonical transformation for which $\overline{\Bff{K}}^{-1} = \Bff{K}$. Then the following are equivalent.
\begin{enumerate}
\item There exists a quadratic form $q:\Bbb{R}^{2n}\to \Bbb{C}$ such that $-\ii q$ is real and $\Bff{K} = \exp H_q$.
\item There exists a quadratic form $q:\Bbb{R}^{2n}\to \Bbb{C}$ such that $\Bff{K} = \exp H_q$ and $\langle u, \ee^{-\ii q^w} u\rangle \geq 0$ for all $u \in L^2(\Bbb{R}^n)$.
\item There exists a quadratic form $q:\Bbb{R}^{2n}\to \Bbb{C}$ such that $\Bff{K} = \exp H_q$ and $\langle u, \ee^{-\ii q^w} u\rangle \leq 0$ for all $u \in L^2(\Bbb{R}^n)$.
\end{enumerate}
\end{proposition}

\begin{proof}
When $\Bff{K}$ is strictly positive, $\overline{\Bff{K}}^{-1} = \Bff{K}$ if and only if $q_1$ from \eqref{eq_q1_Mehler_canonical} is real on $\Bbb{R}^{2n}$: if $\Bff{z} \in \Bbb{R}^{2n}$, then
\[
	\begin{aligned}
	\overline{q_1(\Bff{z})} &= \sigma(\Bff{z}, \ii(1+\overline{\Bff{K}})^{-1}(1-\overline{\Bff{K}}^{-1})\Bff{z})
	\\ &= \sigma(\Bff{z}, -\ii(1+\overline{\Bff{K}}^{-1})^{-1}(1-\overline{\Bff{K}}^{-1})\Bff{z}).
	\end{aligned}
\]
This gives the Hamilton map $H_{\overline{q_1}}$, which is equal to $H_{q_1}$ if and only if $\Bff{K} =  \overline{\Bff{K}}^{-1}$. Furthermore, since $\Bff{K}$ is strictly positive, $q_1(\Bff{z})$ is positive definite.

It is therefore classical that there exists a linear canonical transformation $\Bff{U}:\Bbb{R}^{2n}\to\Bbb{R}^{2n}$ and $\{\mu_j\}_{j=1}^n$ positive real numbers such that, when 
\[
	\Bff{z} = (\Bff{z}_1, \dots,\Bff{z}_n) = (x_1, \xi_1, \dots, x_n, \xi_n),
\]
\[
	(q_1\circ \Bff{U}^{-1})(\Bff{z}) = \sum_{j=1}^n \frac{\mu_j}{2} \Bff{z}_j^2.
\]
Comparing with the harmonic oscillator model $q_0(\Bff{z}) = \frac{1}{2}\Bff{z}^2$, for which $\opnm{Spec} H_{q_0} = \{\pm \ii\}$, we see that $\opnm{Spec}H_{q_1 \circ \Bff{U}^{-1}} = \{\pm \ii \mu_j\}_{j=1}^n$. Because, by \eqref{eq_Hq1_ssym},
\[
	1 \notin \Spec (1+\Bff{K})^{-1}(1-\Bff{K}) = \opnm{Spec}\left(\frac{\ii}{2}H_{q_1}\right)
\]
we have that no $\mu_j$ is equal to 2. Solving \eqref{eq_Hq1_ssym} for $\Bff{K} = \exp H_q$ gives that
\[
	\tanh(H_q/2) = \frac{1}{2\ii} H_{q_1}.
\]
We simplify the process of taking the inverse hyperbolic tangent of $H_{q_1}$ by again using the harmonic oscillator model, for which $\tanh(\tau H_{q_0}) = (\tan \tau) H_{q_0}.$ We pose
\[
	(-\ii q \circ \Bff{U}^{-1})(\Bff{z}) = \sum_{j=1}^n \frac{\tau_j}{2} \Bff{z}_j^2,
\]
which gives the requirement $\tan \tau_j = \frac{1}{2\ii} \mu_j$. We therefore can define $q$ via
\begin{equation}\label{eq_real_log_reduced}
	(-\ii q\circ \Bff{U}^{-1})(\Bff{z}) = \sum_{j=1}^n \opnm{arctanh} (\mu_j/2) \Bff{z}_j^2,
\end{equation}
which may be chosen real if and only if $\mu_j \in (0, 1)$ for all $j$. By \cite[Prop.~2.5]{Viola_2016}, this is equivalent to positivity of $\ee^{-\ii (q\circ \Bff{U}^{-1})^w}$ on $L^2(\Bbb{R}^n)$, and this may be pulled back to positivity of $\ee^{-\ii q^w}$ via a metaplectic transformation. Passing between the conditions on positivity and negativity can be done by replacing $\opnm{arctanh}(\mu_1/2)$ with $\opnm{arctanh}(\mu_1/2) + 2\pi \ii$. This leaves $\exp H_q$ unchanged while reversing the sign of $\ee^{-\ii q^w}$, since in the reduced coordinates \eqref{eq_real_log_reduced} it is equivalent to multiplying by 
\[
	\exp(-\pi \ii(x_1^2 + D_{x_1}^2)) = -1,
\]
as discussed in Remark \ref{rem_ambiguity}.
\end{proof}

\subsection{Egorov relations, Mehler formulas, and the classical-quantum correspondence}\label{ssec_semigroup}

Having set up the equivalent problem on the FBI--Bargmann side, we can readily deduce the Egorov relation and the Mehler formula for $\ee^{-\ii Q}$ via the Egorov relation for the change of variables and the Fourier inversion formula.

\begin{proposition}\label{prop_Mehler_via_FBI}
Let $q$ be any quadratic form for which the canonical transformation $\Bff{K} = \exp H_q$ is strictly positive, let $Q = q^w$, and let $\ee^{-\ii Q}$ be defined as in \eqref{eq_FBI_evolution}. Then the operator $\ee^{-\ii Q}$ is associated with an Egorov relation for polynomial symbols: if $a(x,\xi)$ is a polynomial on $\Bbb{C}^{2n}$, then
\begin{equation}\label{eq_semigroup_Egorov}
	\ee^{-\ii Q}a^w = (a\circ\Bff{K}^{-1})^w(x,D_x)\ee^{-\ii Q}.
\end{equation}
Furthermore, there is a choice of sign such that the Mehler formula gives the Weyl symbol of $\ee^{-\ii Q}$:
\begin{equation}\label{eq_Mehler}
	\ee^{-\ii Q} = \Op^w\left(\frac{\pm 1}{\sqrt{\det\cosh(H_q/2)}}\exp\left(\sigma((x,\xi), \frac{1}{\ii}\tanh(H_q/2)(x,\xi))\right)\right).
\end{equation}
\end{proposition}

\begin{remark}
We prove the Egorov theorem for polynomial symbols because $a\circ \Bff{K}^{-1}$ is generally not defined for $a \in \mathscr{S}'(\Bbb{R}^{2n})$, or even $a \in\mathscr{S}(\Bbb{R}^{2n})$, since $\Bff{K}$ is a complex linear transformation. We recall also that the hypothesis that $\Bff{K}$ is strictly positive is sufficient to apply the reduction in Section \ref{ssec_Fock} by Proposition \ref{prop_positivity_implies_supersymmetry}.
\end{remark}

\begin{proof}
We perform the analysis on the FBI--Bargmann side; see for instance \cite[Section 13.4]{Zworski_2012} for details on the Weyl quantization there. It is classical that, when $G\in GL(n, \Bbb{C})$, the operator
\[
	\mathcal{V}_G = (\det G)^{1/2}u(Gx), \quad u \in \opnm{Hol}(\Bbb{C}^n)
\]
is associated via an Egorov relation with the canonical tranformation $\Bff{V}_G(x,\xi) = (G^{-1}x, G^\top \xi)$. The relation \eqref{eq_semigroup_Egorov} follows by passing to the FBI--Bargmann side, where from \eqref{eq_FBI-side_evolution} we know that $\ee^{-\ii \tilde{Q}} = c\mathcal{V}_{\ee^{-\ii M}}$, which is related to the canonical transformation $\Bff{V}_{\ee^{-\ii M}} = \exp H_{\tilde{q}}$.

As for the Mehler formula, we begin by writing the solution \eqref{eq_FBI_evolution} via the Fourier inversion formula,
\begin{equation}\label{eq_FBI_Fourier_Inversion}
	\ee^{-\ii \tilde{Q}}u(x) = (2\pi)^{-n}\ee^{-\frac{\ii}{2}\tr M}\int \ee^{\ii (\ee^{-\ii M}x - y)\cdot \xi}u(y)\,\dd y \,\dd \xi.
\end{equation}
If $A$ is a matrix for which $1 \notin \opnm{Spec}A$,
\begin{equation}\label{eq_FBI_Weylq_Mehler}
	\begin{aligned}
	\Op^w(\ee^{2 \ii A x\cdot \xi})u(x) &= (2\pi)^{-n}\int \ee^{\ii (x-y)\cdot\xi + \ii A(x+y)\cdot \xi}u(y)\,\dd y\,\dd \xi
	\\ &= (2\pi)^{-n} \int \ee^{\ii((1-A)^{-1}(1+A)x - y)\cdot (1-A^\top)\xi}u(y)\,\dd y\,\dd \xi
	\\ &= \frac{(2\pi)^{-n}}{\det (1-A)}\int \ee^{\ii((1-A)^{-1}(1+A)x - y)\cdot \xi}u(y)\,\dd y\,\dd \xi.
	\end{aligned}
\end{equation}

We pose
\[
	(1-A)^{-1}(1+A) = \ee^{-\ii M}.
\]
Since we have assumed that $\Bff{K}$ is strictly positive, $1 \notin \opnm{Spec}\Bff{K}$. By the similarity relation \eqref{eq_similarity_FBI}, $1 \notin \ee^{-\ii M}$, so we may solve for $A$ to obtain
\begin{equation}\label{eq_FBI_Mehler_matrix}
	A = (1+\ee^{\ii M})^{-1}(1-\ee^{\ii M}).
\end{equation}
The computation
\begin{equation}\label{eq_FBI_cosh_comp}
	1-A = 2\ee^{\ii M}(1+\ee^{\ii M})^{-1}
\end{equation}
shows that $1 \notin \opnm{Spec}A$, as we had supposed. Using \eqref{eq_similarity_FBI} again, we deduce that, if $\tilde{\Bff{K}} = \Bff{TKT}^{-1} = \exp H_{\tilde{q}}$, then
\[
	2\ii Ax\cdot \xi = \ii \sigma((x,\xi), (1+\tilde{\Bff{K}})^{-1}(1-\tilde{\Bff{K}})(x,\xi)).
\]
Note that $\tanh(H_{\tilde{q}}/2) = -(1+\tilde{\Bff{K}})^{-1}(1-\tilde{\Bff{K}})$, so all that remains in proving \eqref{eq_Mehler} on the FBI--Bargmann side is to compute the coefficient.

We have shown that, with $A$ defined in \eqref{eq_FBI_cosh_comp},
\[
	\ee^{-\ii \tilde{Q}} = \ee^{-\frac{\ii}{2}\tr M}\det(1-A)\Op^w(\ee^{2\ii Ax\cdot\xi}).
\]
We therefore compute
\[
	\begin{aligned}
	\ee^{-\frac{\ii}{2}\tr M}\det(1-A) &= 2^{2n}\det\left(\ee^{\frac{\ii}{2}M}(1+\ee^{\ii M})^{-1}\right)
	\\ &= \left(\det \cosh(\ii M/2)\right)^{-1}.
	\end{aligned}
\]
We note also that 
\[
	(\det\cosh(\ii M/2))^2 = \det \cosh(H_{\tilde{q}}/2).
\]
This finishes the proof of \eqref{eq_Mehler} on the FBI--Bargmann side, for $\tilde{q}$ and $\tilde{K}$. The general Mehler formula \eqref{eq_Mehler} follows from the Egorov relation \eqref{eq_Egorov_FBI} for the FBI--Bargmann transform and the similarity relation \eqref{eq_similarity_FBI}.
\end{proof}

\begin{remark}\label{rem_ambiguity}
We emphasize that there is no ambiguity in either $\ee^{-\ii Q}$ or in \eqref{eq_Mehler}, despite making a choice of a square root. Indeed, if we have computed the matrix $M$ associated with $q$ on the FBI--Bargmann side, the correct choice is dictated by $\det\cosh(\ii M/2)$.

The canonical example of this choice (and of the Maslov index) appears with the usual harmonic oscillator $Q_0$ in dimension one, for which
\[
	\ee^{-2\pi \ii Q_0} = -1.
\]
The classical Bargmann transform in Example \ref{ex_Bargmann} reduces $Q_0$ to $x\cdot \partial_x + \frac{1}{2}$, meaning in this case $M = 1$. It is then obvious that $\det \cosh(\pi H_{q_0}) = 1$ while $\cosh(\ii \pi) = -1$.
\end{remark}

We now verify \eqref{eq_Weyl_semigroup} for quadratic forms for which $\exp H_{q_j}$ are positive definite. Fortunately, with the Mehler formula \eqref{eq_Mehler} in hand, verifying this statement is straightforward.

\begin{proposition}\label{prop_classical_quantum}
Let $q_j:\Bbb{R}^{2n}\to\Bbb{C}$ for $j=1,2,3$ be three quadratic forms such that the Hamilton flows $\exp H_{q_j}$ are strictly positive, and let $Q_j = q_j^w$. Then
\[
	\exp H_{q_1}\exp H_{q_2} = \exp H_{q_3} \iff \exists \omega \in \{\pm 1\} \::\: \ee^{-\ii Q_1}\ee^{-\ii Q_2} = \omega \ee^{-\ii Q_3},
\]
\end{proposition}

\begin{proof}
From \cite[Thm.~(5.6), Prop.~(5.12)]{Folland_1989} or \cite[Prop.~5.1]{Viola_2016}, if $T_1, T_2$ are matrices antisymmetric with respect to $\sigma$ and if $\ii \sigma(\Bff{z}, T_j\Bff{z})$ has positive definite real part on $\Bff{z}\in\Bbb{R}^{2n}$, we have the formula
\[
	\Op^w(\ee^{-\ii \sigma(\Bff{z}, T_1\Bff{z})})\Op^w(\ee^{-\ii\sigma(\Bff{z}, T_2\Bff{z})}) = (\det(1 + T_1 T_2))^{-1/2}\Op^w(\ee^{-\ii \sigma(\Bff{z}, \tilde{T}_3\Bff{z})}),
\]
where
\[
	\tilde{T}_3 = 1-(1-T_2)(1+T_1T_2)^{-1}(1-T_1).
\]
To obtain this formula, one computes the sharp product via the Fourier transform of a Gaussian and one uses identities like $\ii \sigma(\Bff{z}, T_j \Bff{z}) = \Bff{z}\cdot \ii \Bff{J}T \Bff{z}$ and $T_j^\top = \Bff{J}T_j\Bff{J}$ for $\Bff{J}(x,\xi) = (-\xi, x)$; we refer the reader to the references for this computation.

Supposing that $T_j = (\Bff{K}_j+1)^{-1}(\Bff{K}_j-1)$ for $j = 1,2$, as is the case when $T_j = \tanh(H_{q_j}/2)$ and $\Bff{K}_j = \exp H_{q_j}$, simplifies this formula even further, particularly because
\[
	1+T_1T_2 = 2(1+\Bff{K}_1)^{-1}(1+\Bff{K}_1\Bff{K}_2)(1+\Bff{K}_2)^{-1}.
\]
We also see that $1 + T_j = 2(1+\Bff{K}_j)^{-1}$, so
\[
	\begin{aligned}
	\tilde{T}_3 &= 1- 2(1+\Bff{K}_2)^{-1}\frac{1}{2}(1+\Bff{K}_2)(1+\Bff{K}_1\Bff{K}_2)^{-1}(1 + \Bff{K}_1)2(1+\Bff{K}_1)^{-1}
	\\ &= 1 - 2(1+ \Bff{K}_1\Bff{K}_2)^{-1}
	\\ &= (\Bff{K}_1\Bff{K}_2+1)^{-1}(\Bff{K}_1 \Bff{K}_2-1).
	\end{aligned}
\]

From \eqref{eq_Mehler}, the fact that $\cosh(H_{q_j}/2) = \frac{1}{2}\ee^{-H_{q_j}/2}(1+\Bff{K}_j)$, and the computations above,
\[
	\begin{aligned}
	\ee^{-\ii Q_1}\ee^{-\ii Q_2} &= \left(\det\left(\cosh(H_{q_1}/2)\cosh(H_{q_2}/2)(1+T_1T_2)\right)\right)^{-1/2}\Op^w(\ee^{\ii \sigma(\Bff{z}, \tilde{T}_3\Bff{z})})
	\\ &= \left(2^{-2n}\det\left(\ee^{H_{q_1}/2}\ee^{H_{q_2}/2}(1+\Bff{K}_1\Bff{K}_2)\right)\right)^{-1/2}\Op^w(\ee^{\ii \sigma(\Bff{z}, \tilde{T}_3\Bff{z})}).
	\end{aligned}
\]
Writing the Mehler formula for $\ee^{-\ii Q_3}$, we see that
\[
	\ee^{-\ii Q_1}\ee^{-\ii Q_2} = \pm \ee^{-\ii -Q_3},
\]
where we have not specified the signs of any of the square roots, if and only $T_3 = \tilde{T}_3$ and if
\[
	\det\left(\ee^{H_{q_1}/2}\ee^{H_{q_2}/2}(1+\Bff{K}_1\Bff{K}_2)\right) = \pm \det\left(\ee^{H_{q_3}/2}(1+\Bff{K}_3)\right).
\]
The former condition holds if and only if $\Bff{K}_1 \Bff{K}_2 = \Bff{K}_3$, which could have been deduced from the Egorov relations. The latter condition is a consequence of the former, because $\det (\Bff{K}_1^{1/2}\Bff{K}_2^{1/2}) = \pm \det ((\Bff{K}_1\Bff{K}_2)^{1/2}).$
\end{proof}

\subsection{Associating a Schr\"odinger evolution to a Gaussian kernel}\label{ssec_proof_gaussian_kernels}

At this point, we can show that Schr\"odinger evolutions of perturbed supersymmetric quadratic forms describe, up to constants, all nondegenerate Gaussian kernels so long as $\det \varphi''_{xy} \neq 0$.

\begin{proof}[Proof of Theorem \ref{thm_gaussian_kernels}]
From \eqref{eq_K_phi_kernel}, one may see that
\[
	\mathfrak{T}_\varphi u(x) = \int \ee^{\ii \varphi(x,y)} u(y)\,\dd y
\]
is associated with the affine canonical transformation 
\[
	\Bff{L}_\varphi \Bff{z} = \Bff{K}_\varphi \Bff{z} + \Bff{w},
\]
where $\Bff{K}_\varphi$ is defined in \eqref{eq_def_K_phi} and
\begin{equation}\label{eq_Gaussian_kernel_shift}
	\Bff{w} = (-(\varphi''_{yx})^{-1}\varphi'_x(0,0), -\varphi''_{xx}(\varphi''_{yx})^{-1} \varphi'_y(0,0) + \varphi'_x(0,0)).
\end{equation}

By Proposition \ref{prop_Gaussian_nondeg_iff_positive}, $\Bff{K}_\varphi$ is strictly positive; by Proposition \ref{prop_supersymmetric_onto}, let $q$ be a quadratic form for which $\exp H_q = \Bff{K}_\varphi$. Since $\Bff{K}_\varphi$ is strictly positive, $1 \notin \opnm{Spec} \Bff{K}_\varphi$, so we may define
\[
	\Bff{v} = (1-\Bff{K}_\varphi)^{-1} \Bff{w}.
\]
The operators $\mathfrak{T}_\varphi$ and $\ee^{-\ii P}$, with $P = \Op^w(q((x,\xi) - \Bff{v}))$, are chosen to correspond to the same canonical transformation \eqref{eq_Egorov_exp-iP}.

By Lemma \ref{lem_shift_Egorov} and Proposition \ref{prop_Mehler_via_FBI}, we can write
\[
	\ee^{-\ii P}u(x) = \int \ee^{\ii \Psi(x,y,\xi)}u(y)\,\dd y\,\dd \xi
\]
for $\Psi$ a degree-2 polynomial. By Proposition \ref{prop_Mehler_integrability}, $\Im \Psi''$ is positive definite. We may therefore integrate out the $\xi$-variables to obtain, for some degree-2 polynomial $\psi(x,y)$,
\[
	\ee^{-\ii P}u(x) = \int \ee^{\ii \psi(x,y)}u(y)\,\dd y = \mathfrak{T}_\psi.
\]
From \eqref{eq_def_K_phi} and \eqref{eq_Gaussian_kernel_shift},  the canonical transformation associated with $\mathfrak{T}_\varphi$ determines the derivative of $\varphi$ and therefore identifies $\varphi$ up to constants. Since $\mathfrak{T}_\varphi$ and $\mathfrak{T}_\psi$ correspond to the same canonical transformation, there exists some $c_0 \in \Bbb{C}$ such that
\[
	\varphi = \psi + c_0.
\]
Setting $c = \ee^{\ii c_0}$ gives
\[
	\mathfrak{T}_\varphi = c\mathfrak{T}_\psi = c\ee^{-\ii P},
\]
which is the statement of the theorem.
\end{proof}

\section{Applications}\label{sec_app}

As an application of the results in this work, we focus principally on the rotated harmonic oscillator. This gives a complete accounting of the possible models in dimension one \cite[Lem.~2.1]{Pravda-Starov_2007} and allows us to visualize the dynamics on phase space associated with subelliptic phenomena and return to equilibrium. As a final example, we show how the classical Bargmann transform can be formally obtained as a non-elliptic Schr\"odinger evolution of a purely imaginary Hamiltonian.

\subsection{The non-self-adjoint harmonic oscillator}

For $\theta \in (-\pi/2, \pi/2)$, let
\begin{equation}\label{eq_def_q_theta}
	q_\theta(x,\xi) = \frac{1}{2}(\ee^{-\ii \theta} \xi^2 + \ee^{\ii\theta} x^2)
\end{equation}
and let
\begin{equation}\label{eq_def_Q_theta}
	Q_\theta = q_\theta^w(x,D_x) = \frac{1}{2}(\ee^{-\ii\theta}D_x^2 + \ee^{\ii\theta}x^2).
\end{equation}
Theorem \ref{thm_norm_quadratic} allows us to find the norm of $\ee^{-\ii tQ_\theta}$ as an operator in $\mathcal{L}(L^2(\Bbb{R}))$, adding to the two expressions found in \cite[Thm.~1.1, 1.2]{Viola_2016}. We remark that all three methods of proof are somewhat different.

\begin{proposition}
Fix $\theta \in (-\pi/2, \pi/2)$ and $t \in \Bbb{C}$. Let $Q_\theta$ be as in \eqref{eq_def_Q_theta}, and define
\begin{equation}\label{eq_RHO_a}
	a = |\cos t|^2 + \cos(2\theta)|\sin t|^2.
\end{equation}
Then $\ee^{-\ii t Q_\theta}$ is compact on $L^2(\Bbb{R})$ if and only if $a > 1$ and $\Im t < 0$, and in this case
\begin{equation}\label{eq_RHO_norm}
	\|\ee^{-\ii t Q_\theta}\| = \left(a - \sqrt{a^2 - 1}\right)^{1/4}.
\end{equation}
\end{proposition}

\begin{proof}
Since
\[
	H_{q_\theta} = \left(\begin{array}{cc} 0 & \ee^{-\ii\theta} \\ -\ee^{\ii\theta} & 0\end{array}\right)
\]
and $H_{q_\theta}^2 = -1$, we compute that the canonical transformation associated with $\ee^{-\ii tQ_\theta}$ is 
\[
	\Bff{K}_1 = \exp(t H_{q_\theta}) = \cos t + H_{q_\theta}\sin t.
\]
The canonical transformation $\Bff{K}_2 = \overline{\Bff{K}_1}^{-1}$ corresponds to $(\ee^{-\ii t Q_\theta})^*= \ee^{\ii \overline{t}Q_{-\theta}}$, and is therefore given by
\[
	\Bff{K}_2 = \overline{\cos t} - \overline{H_{q_\theta}\sin t}.
\]
We obtain
\[
	\begin{aligned}
	\Bff{K}_2\Bff{K}_1 &= |\cos t|^2 + |\sin t|^2\left(\begin{array}{cc} \ee^{2\ii\theta} & 0 \\ 0 & \ee^{-2\ii\theta}\end{array}\right) + 2\ii \Im(H_{q_\theta}\overline{\cos t}\sin t)
	\\ &= \left(\begin{array}{cc} |\cos t|^2 + \ee^{2\ii\theta}|\sin t|^2 & 2\ii \Im(\ee^{-\ii\theta} \overline{\cos t} \sin t) \\ -2\ii \Im(\ee^{\ii\theta} \overline{\cos t} \sin t) & |\cos t|^2 + \ee^{-2\ii\theta}|\sin t|^2\end{array}\right).
	\end{aligned}
\]

The fact that $\Bff{K}_2 \Bff{K}_1$ is canonical implies that $\det(\Bff{K}_2\Bff{K}_1) = 1$. Therefore, when 
\[
	\begin{aligned}
	a &= \frac{1}{2}\tr(\Bff{K}_2\Bff{K}_1)
	\\ &= |\cos t|^2 + \cos(2\theta)|\sin t|^2,
	\end{aligned}
\]
the eigenvalues in Theorem \ref{thm_norm_quadratic} are given by
\[
	\opnm{Spec}(\Bff{K}_2 \Bff{K}_1) = \{a \pm \sqrt{a^2 - 1}\}.
\]
This proves \eqref{eq_RHO_norm}; what remains is to check the strict positivity condition.

As usual, let $\Bff{J}(x,\xi) = (-\xi, x)$. Writing
\[
	B = \ii \Bff{J}(\Bff{K}_2\Bff{K}_1 - 1),
\]
we compute the determinant by computing that
\[
	-2\ii \Im(\ee^{\ii\theta} \overline{\cos t} \sin t) 2\ii \Im(\ee^{-\ii\theta} \overline{\cos t} \sin t) = -2\Re (\overline{\cos^2 t}\sin^2 t) + 2\cos(2\theta)|\sin t \cos t|^2
\]
and that
\begin{multline*}
	(|\cos t|^2 + \ee^{2\ii\theta}|\sin t|^2-1)(|\cos t|^2 + \ee^{-2\ii\theta}|\sin t|^2-1) = (a-1)^2 + \sin^2(2\theta) |\sin t|^4
	\\ = |\cos t|^4 + 2\cos(2\theta)|\sin t \cos t|^2 + |\sin t|^4 - 2a + 1.
\end{multline*}
These together give that
\[
	\begin{aligned}
	\det B &= -2\Re (\overline{\cos^2 t}\sin^2 t) - |\cos t|^4 - |\sin t|^4 + 2a - 1
	\\ &= -\Re\left((\overline{\cos^2 t} + \overline{\sin^2 t})(\cos^2 t + \sin^2 t)\right) + 2a - 1
	\\ &= 2(a-1).
	\end{aligned}
\]

On the other hand,
\[
	\begin{aligned}
	\tr B &= -4(\cos \theta) \Im (\overline{\cos t}\sin t)
	\\ &= -4(\cos\theta)\Im\left(\frac{1}{4\ii}(-2\sinh (2\Im t) + 2\ii \sin(2\Re t))\right)
	\\ &= -2\cos\theta \sinh 2\Im t.
	\end{aligned}
\]
Therefore, by Proposition \ref{prop_compact_strictly_pos}, $\ee^{-\ii t Q}$ is compact if and only if $\opnm{Spec} B \subset \{\lambda > 0\}$, which holds if and only if $\Im t < 0$ and $a > 1$. This completes the proof of the proposition.
\end{proof}

\begin{example}\label{ex_Davies}
One of the principal motivations of this work is to obtain precise information on the behavior of linear perturbations of subelliptic quadratic Hamiltonians, meaning those with positive semidefinite real part for which the Schr\"odinger evolution is compact via some averaging phenomenon. The simplest example is the evolution of the Davies operator \cite[Sec.~14.5]{Davies_2007}
\[
	Q = D_x^2 + \ii x^2 = \ee^{\ii \pi/4} Q_{\pi/4},
\]
for which the semigroup $\{\ee^{-sQ}\}_{s > 0}$ is obviously smoothing. Less obviously, solutions $\ee^{-sQ}u(x)$ for $s > 0$ and $u \in L^2(\Bbb{R})$ are also superexponentially decaying, which can be seen essentially because $\exp (\ii s H_q)$ is strictly positive and therefore $\ee^{-sQ}$ compares favorably with the harmonic oscillator $Q_0$. Specifically, there exists some $C > 0$ such that $\{\ee^{\frac{s^3}{C} Q_0}\ee^{-sQ}\}_{0 \leq s \leq 1/C}$ is a uniformly bounded family in $\mathcal{L}(L^2(\Bbb{R}))$; see, for example, \cite[Sec.~1.2.1]{Aleman_Viola_2014b} or \cite[Prop.~4.1]{Hitrik_Pravda-Starov_Viola_2015b}.

This corresponds with a slow decrease for $\|\ee^{-s Q}\|$ for small positive $s$. Note that when
\[
	t_1 = -t_2 = \frac{1}{\sqrt{2}}s,
\]
then
\[
	\ee^{-sQ} = \ee^{-\ii(t_1 + \ii t_2)Q_{\pi/4}}.
\]
In this case, $a$ from \eqref{eq_RHO_a} is
\[
	\begin{aligned}
	a &= |\cosh(t_1 + \ii t_2)|^2
	\\ &= \cosh^2 t_1 \cos^2 t_2 + \sinh^2 t_1 \sin^2 t_2
	\\ &= 1 + \frac{1}{6}s^4 + \BigO(s^8)
	\end{aligned}
\]
and therefore
\[
	\|\ee^{-sQ}\| = 1 - \frac{1}{4\sqrt{3}}s^2 + \BigO(s^4).
\]

\end{example}

\subsection{The shifted non-self-adjoint harmonic oscillator}

Let
\[
	\Bff{K} = \exp((t_1 + \ii t_2)H_{q_\theta}), \quad t_1, t_2 \in \Bbb{R}.
\]
Suppose that $\Bff{K}$ is strictly positive, and therefore $\det \Im \Bff{K} \neq 0$. If
\[
	A_1 = (\Im \Bff{K})^{-1}(\Re \Bff{K} - 1),
\]
then
\[
	\Bff{a}_1 = \Re \Bff{v} + A_1 \Im \Bff{v}.
\]
A moderately involved but elementary computation reveals that
\[
	A_1 = \frac{1}{a_0}\left(\begin{array}{cc} a_1 & a_2 \\ a_3 & a_4\end{array}\right)
\]
when
\[
	\begin{aligned}
	a_0 &:= \det \Im \Bff{K}_1 = \cos^2 \theta \sinh^2 t_2 - \sin^2 \theta \sin^2 t_1
	\\ a_1 &= \sin t_1 \sinh t_2 - \frac{1}{2}(\sin 2\theta)(\sin^2 t_1 + \sinh^2 t_2)
	\\ a_2 &= (\cos \theta \sinh t_2 + \sin\theta \sin t_1)(\cos t_1 - \cosh t_2)
	\\ a_3 &= (\sin \theta \sin t_1 - \cos \theta \sinh t_2)(\cos t_1 - \cosh t_2)
	\\ a_4 &= \sin t_1 \sinh t_2 + \frac{1}{2}(\sin 2\theta)(\sin^2 t_1 + \sinh^2 t_2)
	\end{aligned}
\]
Note that, because 
\[
	(\ee^{-\ii(t_1 + \ii t_2)Q_\theta})^* = \ee^{-\ii(-t_1 + \ii t_2)Q_{-\theta}},
\]
replacing $\Bff{K}$ by $\overline{\Bff{K}}^{-1} = \exp(-(t_1 - \ii t_2)H_{q_{-\theta}})$ gives that $\Bff{a}_2 = \Re \Bff{v} - A_2 \Im \Bff{v}$ with
\[
	A_2(t_1, t_2, \theta) = A_1(-t_1, t_2, -\theta) = \frac{1}{a_0}\left(\begin{array}{cc} -a_1 & a_2 \\ a_3 & -a_4\end{array}\right).
\]

While the dynamics of the phase-space centers $\Bff{a}_1, \Bff{a}_2$ are moderately complicated, in order to apply Theorem \ref{thm_norm_deg2}, we only need
\[
	A = -A_1 - A_2 = 2\left(\begin{array}{cc} 0 & \frac{\cos t_1 - \cosh t_2}{\sin \theta \sin t_1 - \cos \theta \sinh t_2} \\ \frac{\cos t_1 - \cosh t_2}{\sin \theta \sin t_1 + \cos \theta \sinh t_2} & 0\end{array}\right).
\]
Theorem \ref{thm_norm_deg2} then gives the following relatively simple expression of the influence of a complex phase-space shift on the norm of the Schr\"odinger evolution for a rotated harmonic oscillator.

\begin{proposition}\label{prop_shifted_RHO}
Let $q_\theta$ and $Q_\theta$ be as in \eqref{eq_def_q_theta} and \eqref{eq_def_Q_theta}, fix $\Bff{v} = (v_x, v_\xi) \in \Bbb{C}^{2n}$, and let $t = t_1 + \ii t_2$ for $t_1, t_2 \in \Bbb{R}$ be such that $\exp tH_{q_\theta}$ is strictly positive. Let $P = \Op^w(q_\theta((x,\xi) - \Bff{v}))$. Then, writing the growth factor $G = \|\ee^{-\ii tP}\|/\|\ee^{-\ii tQ_\theta}\|$,
\[
	\log G = \frac{\cos t_1 - \cosh t_2}{\cos \theta \sinh t_2 + \sin \theta \sin t_1}(\Im v_x)^2 + \frac{\cos t_1 - \cosh t_2}{\cos \theta \sinh t_2 - \sin\theta \sin t_1}(\Im v_\xi)^2.
\]
\end{proposition}

\begin{example}\label{ex_shifted_HO}
If $\theta = 0$, then \eqref{eq_intro_norm_SHO} follows from Proposition \eqref{prop_shifted_RHO}. Furthermore, the trajectories associated with the phase-space centers $\Bff{a}_1, \Bff{a}_2$ simplify greatly. We compute that
\[
	\begin{aligned}
	\Im \Bff{K}_1 &= \Im \cos t + H_{q_0}\Im \sin t
	\\ &= \sinh t_2 ( -\sin t_1 + \cos t_1 H_{q_0})
	\\ &= \sinh t_2 H_{q_0}\exp(t_1 H_{q_0}),
	\end{aligned}
\]
and similarly,
\[
	\Re \Bff{K}_1 = \cosh t_2 \exp(t_1 H_{q_0}).
\]
Therefore, 
\[
	\begin{aligned}
	A_1 &= (\Im \Bff{K}_1)^{-1} (\Re \Bff{K}_1 - 1)
	\\ &= -\left(\coth t_2 + \frac{1}{\sinh t_2}\exp(-t_1 H_{q_0})\right)H_{q_0}.
	\end{aligned}
\]
We see that, for $t_2 < 0$ fixed, $\Bff{a}_1 = \Re \Bff{v} + A_1 \Im \Bff{v}$ traces counterclockwise circles (see Figure \ref{f_centers}) of radius $|\Im \Bff{v}|/\sinh t_2$ around the center 
\[
	\Bff{c}_1 = \Bff{c}_1(t_2) = \Re \Bff{v} - (\coth t_2) H_{q_0}\Im \Bff{v},
\]
beginning at $\Bff{c}_1 - \frac{1}{\sinh t_2} H_{q_0}\Im\Bff{v}$. Similarly, $\Bff{a}_2$ traces clockwise circles around $\Bff{c}_2 = \Re \Bff{v} - \coth t_2 H_{q_0}\Im \Bff{v}$. Because the difference $\Bff{a}_2 - \Bff{a}_1$ is always orthogonal to $\Im \Bff{v}$, the contribution to the norm (illustrated in Figure \ref{f_contours}) is simply 
\[
	\exp\left(\frac{1}{2}|\Bff{a}_2 - \Bff{a}_1|\,|\Im \Bff{v}|\right) = \exp\left(\frac{\cos t_1 - \cosh t_2}{\sinh t_2}|\Im \Bff{v}|^2\right).
\]

In addition to this geometric characterization of the norm of $\ee^{-\ii tP_b}$, we can geometrically understand return to equilibrium: as $t_2 \to -\infty$, the centers $\Bff{c}_1$ and $\Bff{c}_2$ tend exponentially quickly towards $\Bff{c}_{1,\infty} = \Re \Bff{v} - H_{q_0}\Im \Bff{v}$ and $\Bff{c}_{2,\infty}= \Re \Bff{v} + H_{q_0}\Im \Bff{v}$; the radius of the circles around these limit centers become exponentially small; the norm of the first spectral projection is the limit $\|\Pi_0\| = \ee^{(\Im \Bff{v})^2}$; and one can even find the ground states of $P^*$ and $P$ by applying shifts corresponding to $\Bff{c}_{1,\infty}$ and $\Bff{c}_{2,\infty}$ to the usual Gaussian $u(x) = \ee^{-x^2/2}$.
\end{example}

\begin{example}\label{ex_shifted_Davies}
As a concrete example of the fragility of the boundedness of the semigroup for a partially elliptic operator, consider for $(w_x, w_\xi) \in \Bbb{R}^2$ the operator
\begin{equation}\label{eq_shifted_Davies}
	P = (D_x - \ii w_\xi)^2 + \ii (x - \ii w_x)^2.
\end{equation}
Note that this is a shift of $Q$ in Example \ref{ex_Davies}; we therefore apply Proposition \ref{prop_shifted_RHO} to the shifted operator with $\theta = \pi/4$ and
\[
	t_1 = -t_2 = \frac{1}{\sqrt{2}}s.
\]
Note that
\[
	\begin{aligned}
	\cos t_1 - \cosh t_2 &= -\frac{1}{2}s^2 + \BigO(s^6)
	\\ \cos \theta \sinh t_2 + \sin \theta \sin t_1 &= -\frac{1}{12}s^3 + \BigO(s^7)
	\\ \cos \theta \sinh t_2 - \sin\theta \sin t_1 &= -s + \BigO(s^5).
	\end{aligned}
\]
We see that we have exponential blowup of $\|\ee^{-sP}\|$ as $s \to 0^+$ only insofar as the perturbation is in the $x$ direction:
\[
	\begin{aligned}
	\log \|\ee^{-sP}\| &= \frac{6}{s}(1+\BigO(s^4))w_x^2 + \frac{s}{2}(1+\BigO(s^4))w_\xi^2 + \log \|\ee^{-sQ}\|
	\\ &= \frac{6}{s}(1+\BigO(s^4))w_x^2 + \frac{s}{2}(1+\BigO(s^4))w_\xi^2 - \frac{1}{4\sqrt{3}}s^2 + \BigO(s^4).
	\end{aligned}
\]
\end{example}

\begin{figure}
  \centering
    \includegraphics[width=1\textwidth]{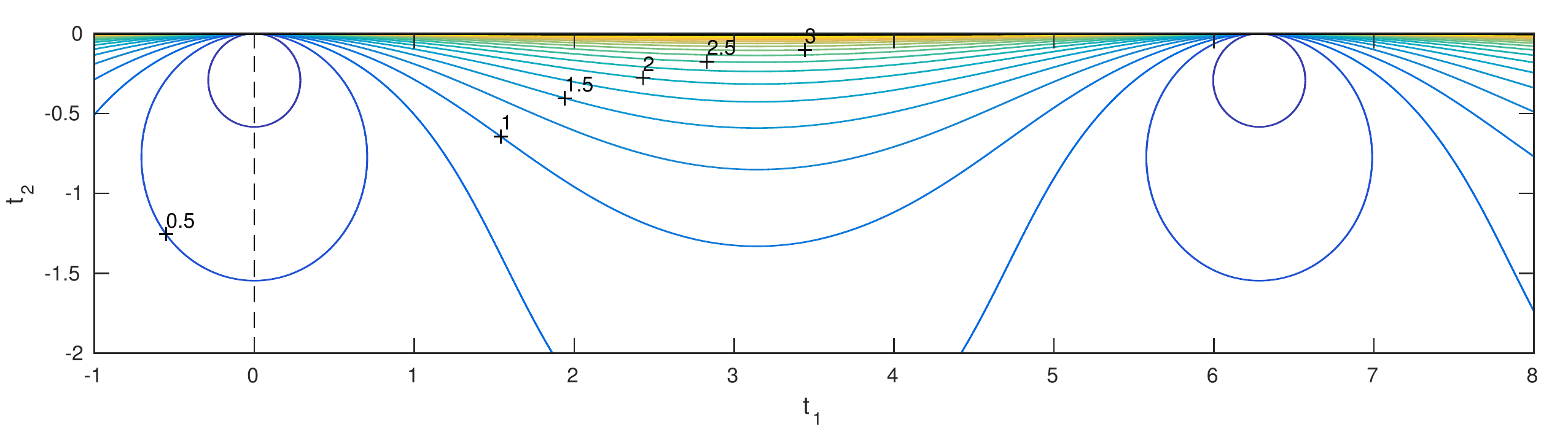}
    \includegraphics[width=1\textwidth]{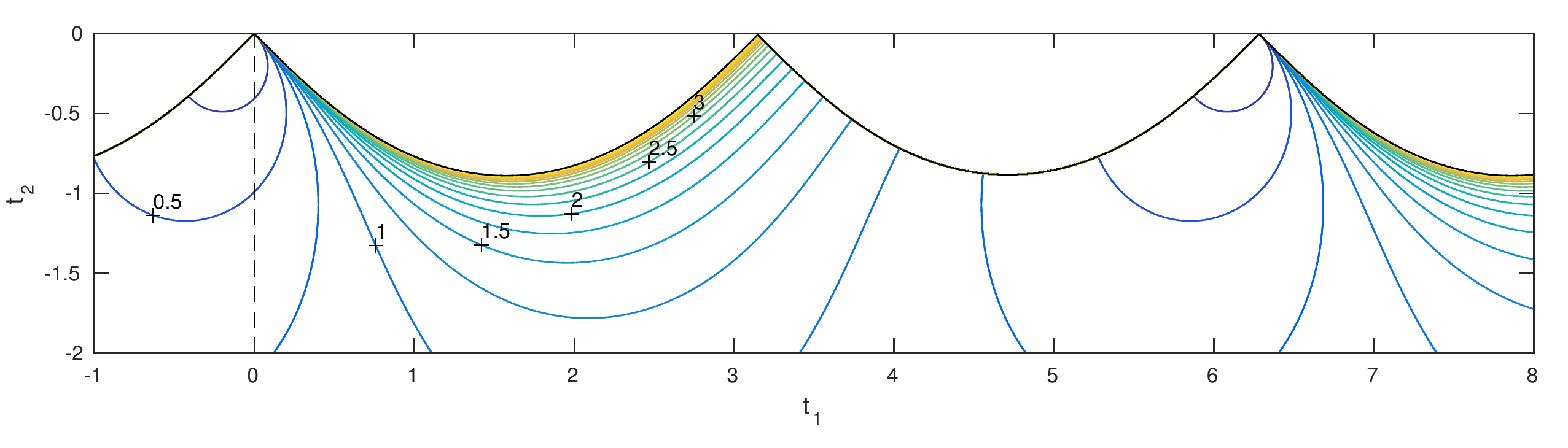}
  \caption{Contours of $\log (\log G + 1)$ for growth factor in Proposition \ref{prop_shifted_RHO} at $t = t_1 + \ii t_2$ with shift $\Bff{v} = \ii(1,0)$ and $\theta = 0$ (above) and $\pi/4$ (below).}
	\label{f_contours}
\end{figure}

\begin{figure}
  \centering
	\includegraphics[width=.4\textwidth]{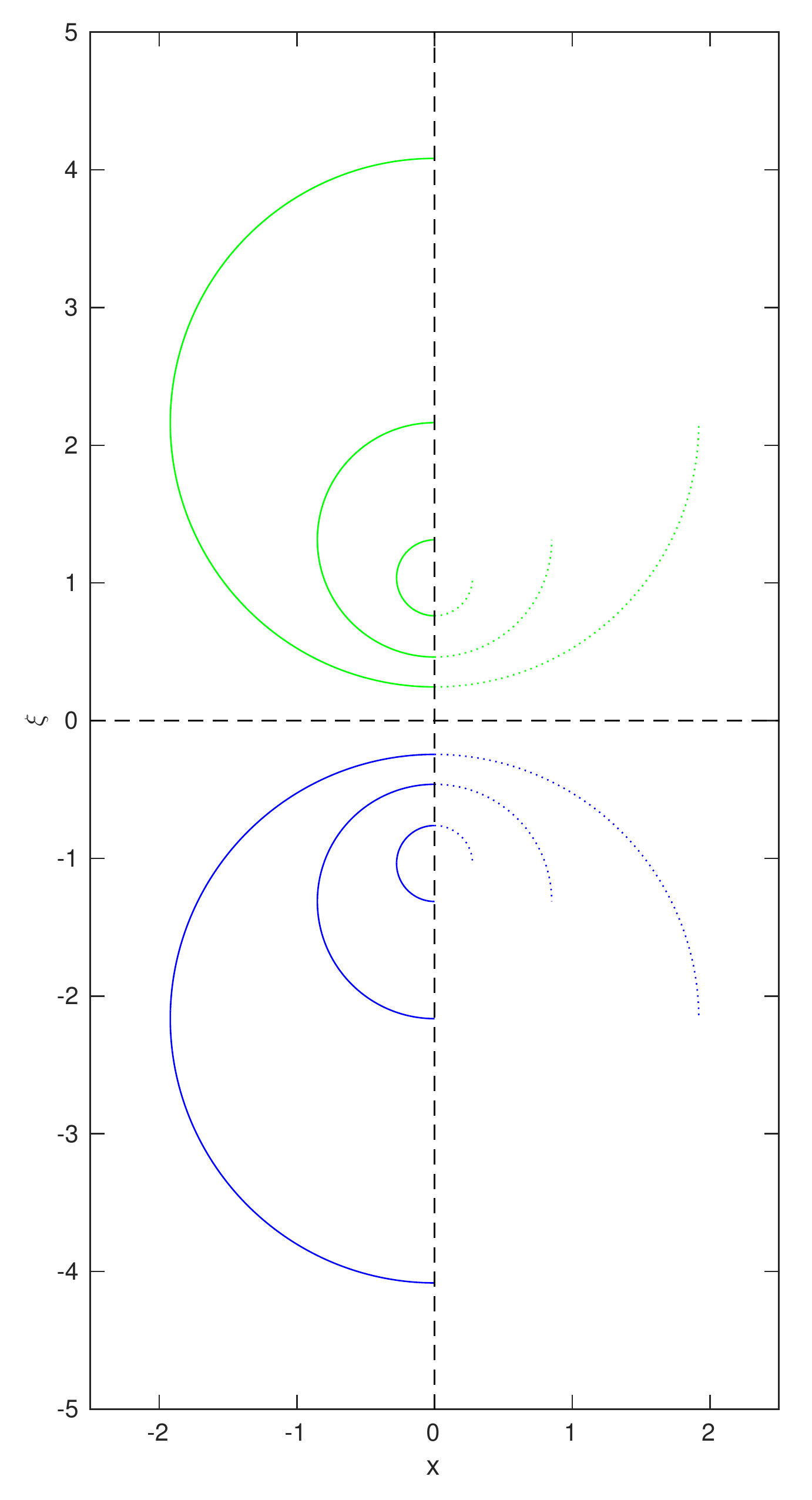}
	\includegraphics[width=.4\textwidth]{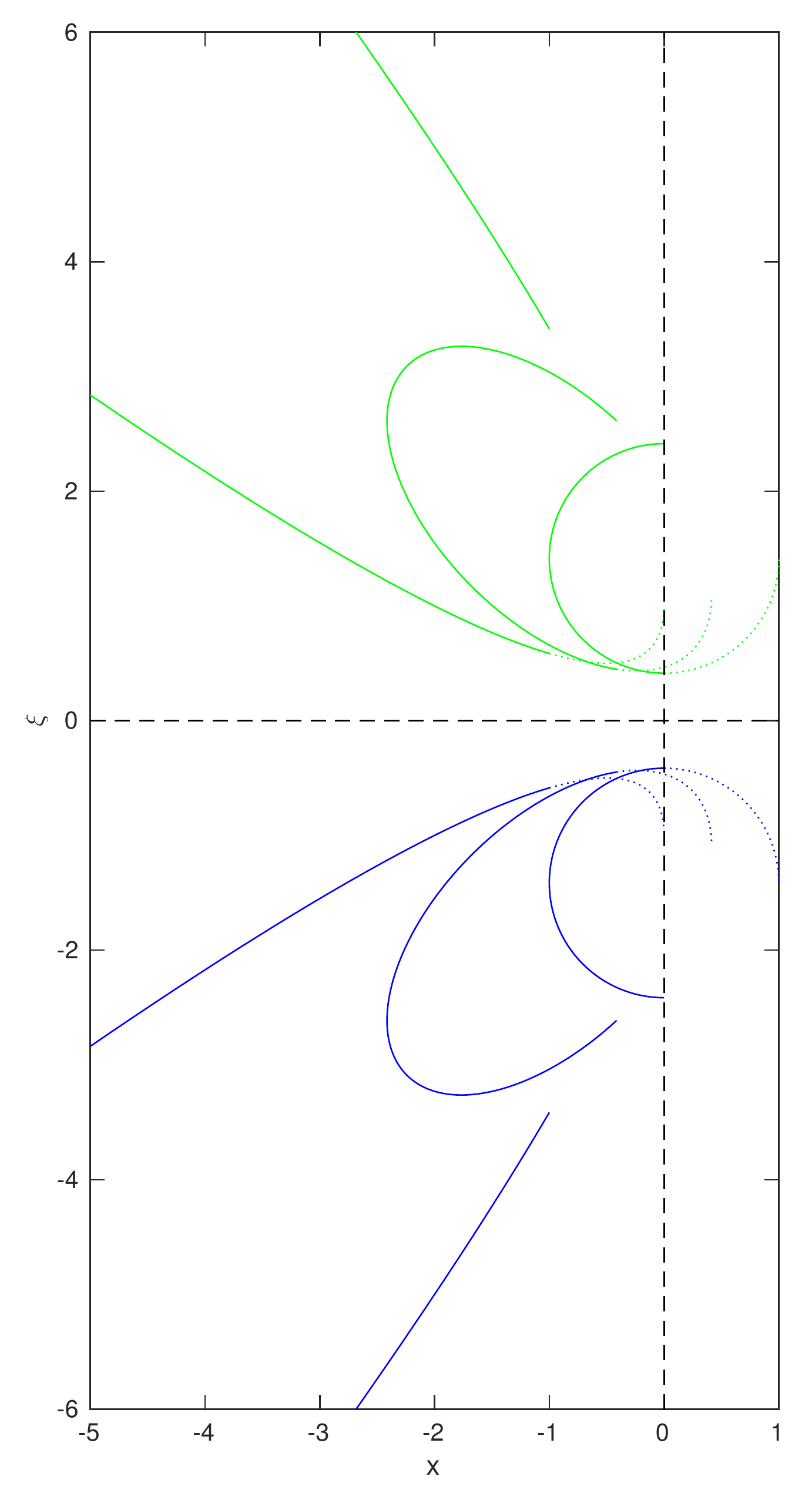}
  \caption{Paths traced by centers $\Bff{a}_1$ below and $\Bff{a}_2$ above for $P$ from \eqref{eq_shifted_Davies}, fixed $t_2$, and $t_1 \in [0, \pi]$ (solid) or $t_1 \in [-\pi/2, 0]$ (dotted). Left: $\theta = 0$ and $t_2 = 0.5, 1, 2$ outside to inside; right: $\theta = 0, \pi/8, \pi/4$ right to left and $t_2 = t_2^c(\pi/4)$ from \eqref{eq_t2_c}.}
	\label{f_centers}
\end{figure}

In Figures \ref{f_contours} and \ref{f_centers}, we illustrate this information. First, in Figure \ref{f_contours}, we draw the contours corresponding to the growth factor for either the shifted harmonic oscillator or the shifted rotated harmonic oscillator with $\pi = 4$. We see that in either case, the symmetry in $t_1$, with period $\pi$, of the norm for the rotated harmonic oscillator \cite{Viola_2016} is broken, and for the rotated harmonic oscillator we see the strong dependence of the norm on the direction in time, corresponding to the choice of a perturbation in the $x$-direction.

In Figure \ref{f_centers}, we draw the paths of $\Bff{a}_1$ and $\Bff{a}_2$ for fixed $t_2$ in various situations. To emphasize the point of departure $t_1 = 0$, we draw $0 \leq t_1 \leq \pi$ as a solid curve and $-\pi/2 \leq t_1 \leq 0$ as a dotted curve. On the left, we have the shifted harmonic oscillator $\theta = 0$. One can see both the exponential explosion of the norm as $t_2 \to 0^-$, owing to increasingly large circles, and the return to equilibrium coming from to exponentially small circles, as $t_2$ decrease. On the left, we have varying values of $\theta$, showing how dependence on the direction in phase space appears as the circle ($\theta = 0$) turns to become an ellipse and then a parabola ($\theta = \pi/4$). We have chosen the critical time
\begin{equation}\label{eq_t2_c}
	t_2^c(\theta) = -\frac{1}{2}\log\left(\frac{1+|\sin \theta|}{1-|\sin \theta|}\right)
\end{equation}
because it marks where the denominators in Proposition 5.4 can go to zero for $t_1 = \pi/2 + \pi k, k \in \Bbb{Z}$. For $t_2 < t_2^c$, fixed, the paths traced by $\Bff{a}_1$ and $\Bff{a}_2$ are bounded. At the same time, $t_2^c$ is the largest value of $t_2$ such that, for all $t_2 < t_2^c$, the operator $\ee^{-\ii(t_1 + \ii t_2)Q_\theta}$ is compact for all $t_1\in\Bbb{R}$. Third and finally, the expansion for $\ee^{-\ii(t_1 + \ii t_2)Q_\theta}$ in eigenfunctions of $Q_\theta$ converges absolutely if and only if  $t_2 < t_2^c$, \cite[Thm.~14.5.1]{Davies_2007} as well as \cite[App.~B]{Krejcirik_Siegl_Tater_Viola_2014} and the references therein.

\subsection{The Bargmann transform via a formal Mehler formula}\label{ssec_Bargmann_via_Mehler}

As a final example, we consider the Bargmann transform itself from Example \ref{ex_Bargmann}. We will see that $\mathfrak{B}_0$ may be \emph{formally} obtained as a Mehler formula along the lines of Proposition \ref{prop_supersymmetric_onto}. This suggests that the link between Hamilton flows and Schr\"odinger evolutions may be pushed far beyond the class of strictly positive Hamilton flows.

The Bargmann transform is chosen to quantize $\Bff{B}_0$ from \eqref{eq_def_B0}. Note that this canonical transformation is not strictly positive: $\overline{\Bff{B}_0}^{-1} = \Bff{B}_0$, and
\[
	\overline{\Bff{B}_0}^{-1}\Bff{B}_0 = -\ii \left(\begin{array}{cc} 0 & 1 \\ 1 & 0\end{array}\right).
\]
Therefore
\[
	\ii \Bff{J}\left(\overline{\Bff{B}_0}^{-1}\Bff{B}_0 - 1\right) = \left(\begin{array}{cc} -1 & \ii \\ -\ii & 1\end{array}\right),
\]
which would be positive definite if $\Bff{B}_0$ were positive, has spectrum $\{\pm \sqrt{2}\}$.

Nonetheless, as in Proposition \ref{prop_supersymmetric_onto}, we define a quadratic form with Hamilton map $\log \Bff{B}_0$. Recalling the harmonic oscillator symbol $q_0(x,\xi) = \frac{1}{2}(x^2 + \xi^2)$, let
\[
	\Bff{U}_0 = \exp (\frac{\pi}{4}H_{q_0}) = \frac{1}{\sqrt{2}} \left(\begin{array}{cc} 1 & 1 \\ -1 & 1\end{array}\right).
\]
This is so that
\[
	\Bff{U}_0 \Bff{B}_0 \Bff{U}_0^{-1} = \left(\begin{array}{cc} \ee^{-\ii \pi/4} & 0 \\ 0 & \ee^{\ii \pi/4}\end{array}\right).
\]
We may find $p_0$ quadratic such that $\Bff{B}_0 = \exp (\frac{\pi}{4}H_{p_0})$ by setting
\[
	H_{p_0} = \frac{4}{\pi}\log \Bff{B}_0 = \ii\Bff{U}_0^{-1} \left(\begin{array}{cc} -1 & 0 \\ 0 & 1\end{array}\right)\Bff{U}_0 = -\ii \left(\begin{array}{cc} 0 & 1 \\ 1 & 0\end{array}\right).
\]
The factor of $\pi/4$ is not essential, but seems to give a pleasant symmetry in formulas \eqref{eq_Bargmann_as_exp_reduces_HO} and \eqref{eq_HO_reduces_Bargmann_as_exp} below.
We obtain $p_0$ from its Hamilton map as
\[
	\begin{aligned}
	p_0(x,\xi) &= \frac{1}{2}\sigma((x,\xi), H_{p_0}(x,\xi))
	\\ &= \frac{\ii}{2}(x^2 - \xi^2).
	\end{aligned}
\]

Naturally, when 
\[
	P_0 = p_0^w = \frac{\ii}{2}(x^2 - D_x^2),
\]
we cannot define $\ee^{-\ii \frac{\pi}{4}P_0}$ by standard functional analysis because $\opnm{Spec}(-\ii P_0) = \Bbb{R}$. Nonetheless, we can write the Mehler formula, re-using the diagonalization of $H_{p_0}$. We begin with
\[
	\det \cosh (tH_{p_0}/2) = \det \left( \Bff{U}_0^{-1} \left(\begin{array}{cc}\cosh(-\ii t/2) & 0 \\ 0 & \cosh(\ii t/2)\end{array} \right)\Bff{U}_0\right) = \cos(t/2)^2.
\]
Since we are working formally (and the constant factor we find is different from that in Example \ref{ex_Bargmann}), we choose the positive sign for the square root in the Mehler formula.

As for the exponent, we note that
\[
	\begin{aligned}
	\frac{1}{\ii} \tanh(tH_{p_0}/2) &= \frac{1}{\ii}\Bff{U}_0^{-1}\left(\begin{array}{cc}-\tanh(\ii t/2) & 0\\ 0 & \tanh(\ii t/2)\end{array}\right)\Bff{U}_0
	\\ &= -\tan(t/2)\Bff{U}_0^{-1}\left(\begin{array}{cc} -1 & 0 \\ 0 & 1\end{array}\right)\Bff{U}_0
	\\ &= -\tan(t/2)\left(\begin{array}{cc} 0 & 1\\ 1 & 0\end{array}\right).
	\end{aligned}
\]
Therefore
\[
	\sigma((x,\xi), \frac{1}{\ii}\tanh(tH_{p_0}/2)(x,\xi)) = -\tan(t/2)(\xi^2 - x^2).
\]

For $t = \pi/4$, or even $t \in (0, \pi)$, this gives a Mehler formula
\[
	M_{tp_0}(x,\xi) = \frac{1}{\cos(t/2)^n}\exp\left(-\tan(t/2)(\xi^2 - x^2)\right),
\]
which is decaying in $\xi$ but is exponentially large as $x \to \infty$. We therefore work formally to integrate out in $\xi$ in the Weyl quantization. Writing $T = \tan(t/2)$ and using elementary trigonometric formulas,
\[
	\begin{aligned}
	M_{tp_0}^w u(x) &= \frac{(2\pi)^{-n}}{\cos(t/2)^n} \iint \ee^{\ii (x-y)\cdot \xi - T(\xi^2 - (\frac{x+y}{2})^2)}u(y)\,\dd y \,\dd \xi
	\\ &= \frac{(2\pi)^{-n}}{\cos(t/2)^n}\left(\frac{\pi}{T}\right)^{n/2}\int \ee^{\frac{1}{4}(T-\frac{1}{T})(x^2 + y^2) + \frac{1}{2}(T+\frac{1}{T})xy}\,u(y)\,\dd y
	\\ &= (\pi \sin t)^{-n/2}\int \ee^{-\frac{1}{2}(\cot t)(x^2 + y^2) + (\sec t)xy}u(y)\,\dd y
	\end{aligned}
\]
We remark that, for $y$ fixed, the kernel is integrable in $\xi$ for $t \in (0, \pi)$, but the resulting integral kernel is integrable in $y$ only when $t \in (0, \pi/2)$.

Setting $t = \pi/4$ gives, formally,
\[
	\begin{aligned}
	\ee^{\frac{\pi}{8}(x^2 - D_x^2)}u(x)  &= \ee^{-\ii \frac{\pi}{4} P_0}u(x)
	\\ &= \left(\frac{\pi}{\sqrt{2}}\right)^{-n/2}\int \ee^{-\frac{1}{2}(x^2 + y^2) + \sqrt{2}xy}u(y)\,\dd y
	\\ &= (2\pi)^{n/4}\mathfrak{B}_0u(x).
	\end{aligned}
\]
We recall that the Egorov relation for $\mathfrak{B}_0$ allows us to reduce the harmonic oscillator $Q_0$, writing, again formally, that
\begin{equation}\label{eq_Bargmann_as_exp_reduces_HO}
	\ee^{-\ii \frac{\pi}{4}P_0}Q_0 = \left(x\cdot \partial_x + \frac{n}{2}\right)\ee^{-\ii \frac{\pi}{4}P_0}.
\end{equation}
What is more, recalling that $\Bff{U}_0 = \exp(\frac{\pi}{4}H_{q_0})$ and that therefore
\[
	H_{p_0\circ \Bff{U}_0^{-1}} = \Bff{U}_0 H_{p_0}\Bff{U}_0^{-1} = \left(\begin{array}{cc} -\ii & 0 \\ 0 & -\ii \end{array}\right),
\]
it is the harmonic oscillator itself which gives a corresponding reduction for $P_0$:
\begin{equation}\label{eq_HO_reduces_Bargmann_as_exp}
	\ee^{-\ii \frac{\pi}{4}Q_0}P_0 = -\left(x\cdot \partial_x + \frac{n}{2}\right)\ee^{-\ii \frac{\pi}{4}Q_0}.
\end{equation}

\bibliographystyle{abbrv}
\bibliography{BibMehler}

\end{document}